\documentclass[12pt]{amsart} 
\usepackage{mathtools}
\usepackage{txfonts}
\usepackage{amssymb}
\usepackage{eucal}
\usepackage{graphicx}
\usepackage{amsmath}
\usepackage{amscd}
\usepackage[all]{xy}           
\usepackage{tikz}
\usepackage{amsfonts,latexsym}
\usepackage{xspace}
\usepackage{epsfig}
\usepackage{float}
\usepackage{color}
\usepackage{fancybox}
\usepackage{colordvi}
\usepackage{multicol}
\usepackage{colordvi}
 \usepackage{ifpdf}
  \ifpdf
   \usepackage[colorlinks,final,backref=page,hyperindex]{hyperref}
  \else
   \usepackage[colorlinks,final,backref=page,hyperindex]{hyperref}
  \fi
\usepackage[active]{srcltx} 
\usepackage{mathrsfs}
\usepackage{graphicx}
\usepackage{array}
\usepackage{makecell}

\usepackage{comment}

\usepackage{yhmath}

\newtheorem{theorem}{Theorem}[section]

\newtheorem{lemma}[theorem]{Lemma}

\newtheorem{prop-def}{Proposition-Definition}[section]
\newtheorem{coro-def}{Corollary-Definition}[section]

\theoremstyle{definition}
\newtheorem{defn}[theorem]{Definition}
\newtheorem{remark}[theorem]{Remark}

\newtheorem{exam}[theorem]{Example}

\newcommand{\nc}{\newcommand}
\nc{\tred}[1]{\textcolor{red}{#1}}
\nc{\tblue}[1]{\textcolor{blue}{#1}}
\nc{\tgreen}[1]{\textcolor{green}{#1}}
\nc{\tpurple}[1]{\textcolor{purple}{#1}}
\nc{\btred}[1]{\textcolor{red}{\bf #1}}
\nc{\btblue}[1]{\textcolor{blue}{\bf #1}}
\nc{\btgreen}[1]{\textcolor{green}{\bf #1}}
\nc{\btpurple}[1]{\textcolor{purple}{\bf #1}}
\nc{\NN}{{\mathbb N}}
\nc{\ncsha}{{\mbox{\cyr X}^{\mathrm NC}}} \nc{\ncshao}{{\mbox{\cyr
X}^{\mathrm NC}_0}}

\newcommand{\efootnote}[1]{}

\renewcommand{\textbf}[1]{}

\newcommand{\delete}[1]{}

\delete{
\nc{\mlabel}[1]{\label{#1} {{\tt {\tiny{(#1)}}}}\ }
\nc{\mcite}[1]{\cite{#1} {{\tiny\tt (#1)}}\ }
\nc{\mref}[1]{\ref{#1}{{\tiny\tt (#1)}}\ }
\nc{\meqref}[1]{~\eqref{#1}{{\tiny\tt (#1)}}\ }
\nc{\mbibitem}[1]{\bibitem[\bf #1]{#1}}
}

	\nc{\mlabel}[1]{\label{#1}}  
	\nc{\mcite}[1]{\cite{#1}}  
	\nc{\mref}[1]{\ref{#1}}  
	\nc{\meqref}[1]{~\eqref{#1}}
	\nc{\mbibitem}[1]{\bibitem{#1}} 

\nc{\tforall}{\quad \text{ for all }}
\nc{\gsb}{Gr\"obner-Shirshov basis\xspace}
\nc{\gsbs}{Gr\"obner-Shirshov bases\xspace}
\nc{\mdl}{\text{dl}}
\nc{\opa}{\ast} \nc{\opb}{\odot} \nc{\op}{\bullet} \nc{\pa}{\frakL}
\nc{\arr}{\rightarrow} \nc{\lu}[1]{(#1)} \nc{\mult}{\mrm{mult}}
\nc{\diff}{\mathfrak{Diff}}
\nc{\opc}{\sharp}\nc{\opd}{\natural}
\nc{\ope}{\circ}
\nc{\dpt}{\mathrm{d}}
\nc{\hck}{H_{\rm RT}}
\nc{\vdf}{\calf}
\nc{\ldf}{\calf_\ell}
\nc{\hlf}{H_\ell}
\nc{\onek}{\mathbf{1}_\bfk}

\nc{\merba}{matching extended Rota-Baxter algebra\xspace}
\nc{\Merba}{Matching extended Rota-Baxter algebra\xspace}
\nc{\merbas}{matching extended Rota-Baxter algebras\xspace}
\nc{\Merbas}{Matching extended Rota-Baxter algebras\xspace}

\nc{\match}{matching\xspace}
\nc{\Match}{Matching\xspace}
\nc{\me}{matching extended\xspace}
\nc{\Me}{Matching extended\xspace}

\nc{\paybe}{polarized associative Yang-Baxter equation\xspace}
\nc{\Paybe}{Polarized associative Yang-Baxter equation\xspace}
\nc{\cpaybe}{PAYBE}

\nc{\erba}{extended Rota-Baxter algebra\xspace}
\nc{\erbas}{extended Rota-Baxter algebras\xspace}

\nc{\diam}{alternating\xspace}
\nc{\Diam}{Alternating\xspace}
\nc{\cdiam}{canonical alternating\xspace}
\nc{\Cdiam}{Canonical alternating\xspace}
\nc{\AW}{\mathcal{A}}

\nc{\ari}{\mathrm{ar}}

\nc{\lef}{\mathrm{lef}}

\nc{\Sh}{\mathrm{ST}}

\nc{\Cr}{\mathrm{Cr}}

\nc{\st}{{Schr\"oder tree}\xspace}
\nc{\sts}{{Schr\"oder trees}\xspace}

\nc{\vertset}{\Omega} 

\nc{\assop}{\quad \begin{picture}(5,5)(0,0)
\line(-1,1){10}
\put(-2.2,-2.2){$\bullet$}
\line(0,-1){10}\line(1,1){10}
\end{picture} \quad \smallskip}

\nc{\operator}{\begin{picture}(5,5)(0,0)
\line(0,-1){6}
\put(-2.6,-1.8){$\bullet$}
\line(0,1){9}
\end{picture}}

\nc{\idx}{\begin{picture}(6,6)(-3,-3)
\put(0,0){\line(0,1){6}}
\put(0,0){\line(0,-1){6}}
\end{picture}}

\nc{\pb}{{\mathrm{pb}}}
\nc{\Lf}{{\mathrm{Lf}}}

\nc{\lft}{{left tree}\xspace}
\nc{\lfts}{{left trees}\xspace}

\nc{\fat}{{fundamental averaging tree}\xspace}

\nc{\fats}{{fundamental averaging trees}\xspace}
\nc{\avt}{\mathrm{Avt}}

\nc{\rass}{{\mathit{RAss}}}

\nc{\aass}{{\mathit{AAss}}}

\nc{\vin}{{\mathrm Vin}}    
\nc{\lin}{{\mathrm Lin}}    
\nc{\inv}{\mathrm{I}n}
\nc{\gensp}{V} 
\nc{\genbas}{\mathcal{V}} 
\nc{\bvp}{V_P}     
\nc{\gop}{{\,\omega\,}}     

\nc{\bin}[2]{ (_{\stackrel{\scs{#1}}{\scs{#2}}})}  
\nc{\binc}[2]{\left (\!\! \begin{array}{c} \scs{#1}\\
    \scs{#2} \end{array}\!\! \right )}  
\nc{\bincc}[2]{\left ( {\scs{#1} \atop
    \vspace{-1cm}\scs{#2}} \right )}  
\nc{\bs}{\bar{S}} \nc{\cosum}{\sqsubset} \nc{\la}{\longrightarrow}
\nc{\rar}{\rightarrow} \nc{\dar}{\downarrow} \nc{\dprod}{**}
\nc{\dap}[1]{\downarrow \rlap{$\scriptstyle{#1}$}}
\nc{\md}{\mathrm{dth}} \nc{\uap}[1]{\uparrow
\rlap{$\scriptstyle{#1}$}} \nc{\defeq}{\stackrel{\rm def}{=}}
\nc{\disp}[1]{\displaystyle{#1}} \nc{\dotcup}{\
\displaystyle{\bigcup^\bullet}\ } \nc{\gzeta}{\bar{\zeta}}
\nc{\hcm}{\ \hat{,}\ } \nc{\hts}{\hat{\otimes}}
\nc{\barot}{{\otimes}} \nc{\free}[1]{\bar{#1}}
\nc{\uni}[1]{\tilde{#1}} \nc{\hcirc}{\hat{\circ}} \nc{\lleft}{[}
\nc{\lright}{]} \nc{\lc}{\lfloor} \nc{\rc}{\rfloor}
\nc{\curlyl}{\left \{ \begin{array}{c} {} \\ {} \end{array}
    \right .  \!\!\!\!\!\!\!}
\nc{\curlyr}{ \!\!\!\!\!\!\!
    \left . \begin{array}{c} {} \\ {} \end{array}
    \right \} }
\nc{\longmid}{\left | \begin{array}{c} {} \\ {} \end{array}
    \right . \!\!\!\!\!\!\!}
\nc{\onetree}{\bullet} \nc{\ora}[1]{\stackrel{#1}{\rar}}
\nc{\ola}[1]{\stackrel{#1}{\la}}
\nc{\ot}{\otimes} \nc{\mot}{{{\boxtimes\,}}}
\nc{\otm}{\overline{\boxtimes}} \nc{\sprod}{\bullet}
\nc{\scs}[1]{\scriptstyle{#1}} \nc{\mrm}[1]{{\rm #1}}
\nc{\margin}[1]{\marginpar{\rm #1}}   
\nc{\dirlim}{\displaystyle{\lim_{\longrightarrow}}\,}
\nc{\invlim}{\displaystyle{\lim_{\longleftarrow}}\,}
\nc{\mvp}{\vspace{0.3cm}} \nc{\tk}{^{(k)}} \nc{\tp}{^\prime}
\nc{\ttp}{^{\prime\prime}} \nc{\svp}{\vspace{2cm}}
\nc{\vp}{\vspace{8cm}} \nc{\proofbegin}{\noindent{\bf Proof: }}
\nc{\proofend}{$\blacksquare$ \vspace{0.3cm}}
\nc{\modg}[1]{\!<\!\!{#1}\!\!>}
\nc{\intg}[1]{F_C(#1)} \nc{\lmodg}{\!
<\!\!} \nc{\rmodg}{\!\!>\!}
\nc{\cpi}{\widehat{\Pi}}
\nc{\sha}{{\mbox{\cyr X}}}  

\newfont{\scyr}{wncyr10 scaled 550}
\nc{\ssha}{\mbox{\bf \scyr X}}
\nc{\shap}{{\mbox{\cyrs X}}} 
\nc{\shpr}{\diamond}    
\nc{\shp}{\ast} \nc{\shplus}{\shpr^+}
\nc{\shprc}{\shpr_c}    
\nc{\msh}{\ast} \nc{\zprod}{m_0} \nc{\oprod}{m_1}
\nc{\vep}{\epsilon} \nc{\labs}{\mid\!} \nc{\rabs}{\!\mid}
\nc{\sqmon}[1]{\langle #1\rangle}

\nc{\mmbox}[1]{\mbox{\ #1\ }} \nc{\dep}{\mrm{dep}} \nc{\fp}{\mrm{FP}}
\nc{\rchar}{\mrm{char}} \nc{\End}{\mrm{End}} \nc{\Fil}{\mrm{Fil}}
\nc{\Mor}{Mor\xspace} \nc{\gmzvs}{gMZV\xspace}
\nc{\gmzv}{gMZV\xspace} \nc{\mzv}{MZV\xspace}
\nc{\mzvs}{MZVs\xspace} \nc{\Hom}{\mrm{Hom}} \nc{\id}{\mrm{id}}
\nc{\im}{\mrm{im}} \nc{\incl}{\mrm{incl}} \nc{\map}{\mrm{Map}}
\nc{\mchar}{\rm char} \nc{\nz}{\rm NZ} \nc{\supp}{\mathrm Supp}

\nc{\Alg}{\mathbf{Alg}} \nc{\Bax}{\mathbf{Bax}} \nc{\bff}{\mathbf f}
\nc{\bfk}{{\bf k}} \nc{\bfone}{{\bf 1}} \nc{\bfx}{\mathbf x}
\nc{\bfy}{\mathbf y}
\nc{\base}[1]{\bfone^{\otimes ({#1}+1)}} 
\nc{\Cat}{\mathbf{Cat}}

\nc{\detail}{\marginpar{\bf More detail}
    \noindent{\bf Need more detail!}
    \svp}
\nc{\Int}{\mathbf{Int}} \nc{\Mon}{\mathbf{Mon}}
\nc{\rbtm}{{shuffle }} \nc{\rbto}{{Rota-Baxter }}
\nc{\remarks}{\noindent{\bf Remarks: }} \nc{\Rings}{\mathbf{Rings}}
\nc{\Sets}{\mathbf{Sets}} \nc{\wtot}{\widetilde{\odot}}
\nc{\wast}{\widetilde{\ast}} \nc{\bodot}{\bar{\odot}}
\nc{\bast}{\bar{\ast}} \nc{\hodot}[1]{\odot^{#1}}
\nc{\hast}[1]{\ast^{#1}} \nc{\mal}{\mathcal{O}}
\nc{\tet}{\tilde{\ast}} \nc{\teot}{\tilde{\odot}}
\nc{\oex}{\overline{x}} \nc{\oey}{\overline{y}}
\nc{\oez}{\overline{z}} \nc{\oef}{\overline{f}}
\nc{\oea}{\overline{a}} \nc{\oeb}{\overline{b}}
\nc{\weast}[1]{\widetilde{\ast}^{#1}}
\nc{\weodot}[1]{\widetilde{\odot}^{#1}} \nc{\hstar}[1]{\star^{#1}}
\nc{\lae}{\langle} \nc{\rae}{\rangle}
\nc{\lf}{\lfloor}
\nc{\rf}{\rfloor}

\nc{\QQ}{{\mathbb Q}}
\nc{\RR}{{\mathbb R}} \nc{\ZZ}{{\mathbb Z}}

\nc{\cala}{{\mathcal A}} \nc{\calb}{{\mathcal B}}
\nc{\calc}{{\mathcal C}}
\nc{\cald}{{\mathcal D}} \nc{\cale}{{\mathcal E}}
\nc{\calf}{{\mathcal F}} \nc{\calg}{{\mathcal G}}
\nc{\calh}{{\mathcal H}} \nc{\cali}{{\mathcal I}}
\nc{\call}{{\mathcal L}} \nc{\calm}{{\mathcal M}}
\nc{\caln}{{\mathcal N}} \nc{\calo}{{\mathcal O}}
\nc{\calp}{{\mathcal P}} \nc{\calr}{{\mathcal R}}
\nc{\cals}{{\mathcal S}} \nc{\calt}{{\mathcal T}}
\nc{\calu}{{\mathcal U}} \nc{\calw}{{\mathcal W}} \nc{\calk}{{\mathcal K}}
\nc{\calx}{{\mathcal X}} \nc{\CA}{\mathcal{A}}

\nc{\fraka}{{\mathfrak a}} \nc{\frakA}{{\mathfrak A}}
\nc{\frakb}{{\mathfrak b}} \nc{\frakB}{{\mathfrak B}}
\nc{\frakc}{{\mathfrak c}}
\nc{\frakD}{{\mathfrak D}} \nc{\frakF}{\mathfrak{F}}
\nc{\frakf}{{\mathfrak f}} \nc{\frakg}{{\mathfrak g}}
\nc{\frakH}{{\mathfrak H}} \nc{\frakL}{{\mathfrak L}}
\nc{\frakM}{{\mathfrak M}} \nc{\bfrakM}{\overline{\frakM}}
\nc{\frakm}{{\mathfrak m}} \nc{\frakP}{{\mathfrak P}}
\nc{\frakN}{{\mathfrak N}} \nc{\frakp}{{\mathfrak p}}
\nc{\frakS}{{\mathfrak S}} \nc{\frakT}{\mathfrak{T}}
\nc{\frakX}{{\mathfrak X}}
\nc{\BS}{\mathbb{S
}}

\font\cyr=wncyr10 \font\cyrs=wncyr7

\nc{\ID}{{\rm I}}\nc{\lbar}[1]{\overline{#1}}\nc{\bre}{{\rm bre}}
\nc{\sd}{\cals}\nc{\rb}{\rm RB}\nc{\A}{\rm A}\nc{\LL}{\rm L}\nc{\tx}{\tilde{X}}
\nc{\col}{\Delta_{\rm RT}}
\nc{\mul}{m_{\rm RT}}
\nc{\ul}{u_{\rm RT}}
\nc{\epl}{\varepsilon_{\rm RT}}
\nc{\hl}{H_{\rm RT}}
\nc{\arro}[1]{#1}\nc{\px}{P_{\tx}}\nc{\pw}{P_{\mathfrak{w}}}\nc{\pl}{B_\omega^+}
\nc{\pp}{\pl}\nc{\ppp}[1]{B^+(#1)}\nc{\dw}{\diamond_{\mathfrak{w}}}\nc{\dl}{\diamond_{\rm \ell}}
\nc{\ncshaw}{\sha^{{\rm NC}}_{\Omega}}
\nc{\ncshal}{\sha^{{\rm NC}}_{{\rm RT}}}
\nc{\ncshall}{\sha^{{\rm NC}}_{{\rm RT,\,(\lambda,\,\lambda)}}}
\nc{\ver}{\rm V}\nc{\ld}{l}\nc{\del}{\Delta_{{\rm \ell}}}\nc{\epsl}{\epsilon_{{\rm \ell}}}
\nc{\uul}{u_{{\rm \ell}}}\nc{\oneh}{\mathbf{1}}\nc{\onew}{\mathbf{1}}
\nc{\etree}{1} \nc{\conc}{m_{\rm RT}} \nc{\mpu}{u^{\ast}} \nc{\mpv}{v^{\ast}}
\nc{\brep}{\text{bre}_{P}} \nc{\leqo}{\leq_{\text{db}}} \nc{\odb}{<_{\text{db}}}

\nc{\hrtb}{H_{\rm RT}(X\sqcup\Omega)}
\nc{\hrts}{H_{\rm RT}(X, \Omega)}
\nc{\rts}{\mathcal{T}_{\rm RT}(X, \Omega)}
\nc{\rfs}{\mathcal{F}_{\rm RT}(X, \Omega)}
\nc{\ldl}{\leq_{\mathrm{db}}}
\nc{\pla}{B_{\alpha}^{+}}
\nc{\plb}{B_{\beta}^{+}}

\nc{\bim}[1]{#1}  \nc{\shaop}{\sha_{\Omega}^{+}}  \nc{\shao}{\sha_{\Omega}}
\nc{\bbim}[2]{#1 #2} \nc{\bbbim}[2]{#1,\, #2} \nc{\RBF}{{\rm RBF}}
\nc{\frbf}{F_{\RBF}} \nc{\shaf}{\ssha_{\tiny{\Omega}}} \nc{\sham}{\diamond_{\Omega}}

\nc{\dnx}{\Delta_n A} \nc{\dx}{\Delta A} \nc{\dgp}{{\rm deg_{P}}}
\nc{\dgt}{{\rm deg_{T}}} \nc{\dg}{{\rm deg}} \nc{\ida}{ID($A$)} \nc{\tu}{\tilde{u}} \nc{\tv}{\tilde{v}}
 \nc{\fbase}{\calb} \nc{\LF}{\mathrm{RF}} \nc{\FFA}{\mathrm{LF}} \nc{\irr}{\mathrm{Irr}}
 \nc{\result}{\bfk\mathrm{Irr}(S_n)}  \nc{\I}{I_{\mathrm{ID},n}^0}
 \nc{\nrs}{\calr_n^\star} \nc{\ii}{\mathrm{I}} \nc{\iii}{\mathrm{II}}
\nc{\intl}{{\rm int}}\nc{\ws}[1]{{#1}}\nc{\deleted}[1]{\delete{#1}}\nc{\plas}{placements\xspace}

\nc{\Id}{\mathrm{Id}} \nc{\Irr}{\mathrm{Irr}}

\newcommand{\sbullet}
{\begin{picture}(5,5)(-2,-1)
		\put(1,2){\circle*{2}}
\end{picture}}

\nc{\tos}{totally ordered set }
\nc{\nes}{nonempty set}
\nc{\Po}{(P_\omega)_{\omega\in \Omega}}

\nc{\Pop}{(P'_\omega)_{\omega\in \Omega}}

\nc{\Bo}{(B_{\omega}^+)_{\omega\in \Omega}}

\nc{\leafset}{\phi}
\nc{\leafsp}{\Phi}

\newcommand{\tdun}[1]  
{\begin{picture}(10,5)(-2,-1)
\put(2,3){\circle*{2}}
\put(3,-2){\tiny #1}
\end{picture}}

\newcommand{\tddeux}[2]{\begin{picture}(12,5)(0,-1)
\put(3,0){\circle*{2}}
\put(3,0){\line(0,1){5}}
\put(3,5){\circle*{2}}
\put(6,-3){\tiny #1}
\put(6,3){\tiny #2}
\end{picture}}

\newcommand{\tdtroisun}[3]{\begin{picture}(20,12)(-5,-1)
\put(3,0){\circle*{2}}
\put(-0.65,0){$\vee$}
\put(6,7){\circle*{2}}
\put(0,7){\circle*{2}}
\put(5,-2){\tiny #1}
\put(8,5){\tiny #2}
\put(-6,5){\tiny #3}
\end{picture}}
\newcommand{\tdtroisdeux}[3]{\begin{picture}(12,12)(-2,-1)
\put(0,0){\circle*{2}}
\put(0,0){\line(0,1){5}}
\put(0,5){\circle*{2}}
\put(0,5){\line(0,1){5}}
\put(0,10){\circle*{2}}
\put(3,-2){\tiny #1}
\put(3,3){\tiny #2}
\put(3,9){\tiny #3}
\end{picture}}

\newcommand{\tdquatretrois}[4]{\begin{picture}(20,20)(-5,-1)
\put(3,0){\circle*{2}}
\put(-.65,0){$\vee$}
\put(6,7){\circle*{2}}
\put(0,7){\circle*{2}}
\put(6,14){\circle*{2}}
\put(6,7){\line(0,1){7}}
\put(5,-2){\tiny #1}
\put(8,5){\tiny #2}
\put(-6,5){\tiny #4}
\put(8,12){\tiny #3}
\end{picture}}
\newcommand{\tddquatretrois}[4]{\begin{picture}(20,20)(-5,-1)
\put(3,0){\circle*{2}}
\put(-.65,0){$\vee$}
\put(6,7){\circle*{2}}
\put(0,7){\circle{2}}
\put(6,14){\circle{2}}
\put(6,7){\line(0,1){7}}
\put(5,-2){\tiny #1}
\put(8,5){\tiny #2}
\put(-28,10){\tiny #4}
\put(8,12){\tiny #3}
\end{picture}}

\definecolor{red}{rgb}{1.,0.,0.}
\definecolor{green}{rgb}{0.,1.,0.}
\definecolor{blue}{rgb}{0.,0.,1.}

\begin{document}

\title[Generalized Connes-Kreimer Hopf algebras by weighted cocycles]{Generalized Connes-Kreimer Hopf algebras on decorated rooted forests by weighted cocycles}

\author{Fei Wang}
\address{School of Mathematics and Statistics,
	Nanjing University of Information Science \& Technology, Nanjing, Jiangsu 210044, China}
\email{wangfei@nuist.edu.cn}

\author{Li Guo}
\address{Department of Mathematics and Computer Science, Rutgers University, Newark, NJ 07102, USA}
\email{liguo@rutgers.edu}

\author{Yi Zhang}
\address{School of Mathematics and Statistics,
	Nanjing University of Information Science \& Technology, Nanjing, Jiangsu 210044, China}
\email{zhangy2016@nuist.edu.cn}

\date{\today}
\begin{abstract}
The Connes–Kreimer Hopf algebra of rooted trees is an operated Hopf algebra whose coproduct satisfies the classical Hochschild 1-cocycle condition. In this paper, we extend the setting from rooted trees to the space $H_{\rm RT}(X,\Omega)$ of $(X,\Omega)$-rooted trees, in which internal vertices are decorated by a set $\Omega$ and leafs are decorated by $X \cup \Omega$. We introduce a new coalgebra structure on $H_{\rm RT}(X,\Omega)$ whose coproduct satisfies a weighted Hochschild 1-cocycle condition involving multiple operators, thereby generalizing the classical condition. A combinatorial interpretation of this coproduct is also provided. We then endow $H_{\rm RT}(X,\Omega)$ with a Hopf algebra structure. Finally, we define weighted $\Omega$-cocycle Hopf algebras, characterized by a Hochschild 1-cocycle condition with weights, and show that $H_{\rm RT}(X,\Omega)$ is the free object in the category of $\Omega$-cocycle Hopf algebras.
\end{abstract}

\subjclass[2020]{    
16T30,	
05E16,   
16T10, 
05C05,   
16W99, 
16S10 
}

\keywords{decorated rooted tree; Hopf algebra; bialgebra; 1-cocycle condition with weight}

\maketitle

\vspace{-1cm}

\tableofcontents  

\vspace{-1cm}
\setcounter{section}{0}  
\allowdisplaybreaks  

\section{Introduction}
This paper generalizes the cocycle property in the Connes-Kreimer Hopf algebra of rooted trees to a weighted cocycle condition in a Hopf algebra of multi-decorated rooted trees with multiple operators, and obtain the universal property of the new Hopf algebra as weighted cocycle Hopf algebras. 

\subsection{Hopf algebras of decorated rooted trees}
In the landmark work of Connes and Kreimer \mcite{CK98} on the renormalization of perturbative quantum field theory, a Hopf algebra of rooted tree was introduced as a toy model of the Hopf algebra of Feynman graphs\,\mcite{CK1,Kr98}. 
Other types of rooted tree Hopf algebras have also established extensive connections with other areas of mathematics, including operad theory \mcite{LR98,LR04}, pre-Lie and Lie algebras \mcite{Mu06}, Rota-Baxter algebras \mcite{EG08, ZGG16, ZGG, ZL25}, differential algebras \mcite{GL05}, and infinitesimal bialgebras \mcite{ZCGL,ZG20}. Furthermore, rooted tree Hopf algebras find applications in diverse interdisciplinary areas such as numerical analysis \mcite{LM09}, stochastic processes \mcite{BHZ19}, and minimalism program of generative linguistics \mcite{MBC23}.

In order to adapt to more complex combinatorial and physics applications, decorations on rooted tree Hopf algebras have been gradually introduced. For example, Foissy decorated the vertices of planar rooted trees with a nonempty set $\Omega$ in \mcite{Fo02}, and constructed a noncommutative and noncocommutative rooted tree Hopf algebra based on these vertex decorated rooted trees. 
In the context operated algebras\,\mcite{Guo09}, two decorations on rooted trees are introduced, one on the internal vertices and one on the leaf vertices. They give the construction of free operated nonunitary semigroups and free operated nonunitary algebras. These decorations are modified in ~\mcite{ZGG16}, making it possible to construct free operated monoids and free operated unitary algebras with multiple grafting operators.
Building on \mcite{ZGG16} and generalizing Foissy's noncommutative rooted tree Hopf algebra, the paper \mcite{ZGG}  constructed a Hopf algebra of rooted forests decorated by two disjoint sets $X$ and $\Omega$. 

\subsection{Grafting operators on rooted tree Hopf algebras and cocycle conditions}
The rich theory and combinatorial significance of rooted tree Hopf algebras stemmed from the grafting of a forest to a tree, making the Hopf algebra into an operated Hopf algebra for which the cocycle satisfies a 1-cocycle condition. 

The study in this direction can be traced back to Kurosh \mcite{Kur60}, who introduced the concept of algebras with linear operators, influenced by Higgins' work on multi-operator groups \mcite{Hi56}. 
The realizations of free operated semigroups and free operated nonunitary algebras on rooted trees were provided in\, \mcite{Guo09}; while realizations of free operated monoids and free operated (unitary) algebras on rooted trees were provided in\, \mcite{ZGG16}. The grafting operator of rooted forests equipped the underlying algebra of the Connes-Kreimer Hopf algebra with an operated algebra structure which is in fact a free operated algebra~\mcite{Guo09,ZGG16}. On the other hand, the interaction of this operator with the comultiplication leads to the classical Hochschild 1-cocycle condition\,\mcite{Mo}, again satisfying a universal property. 
In \mcite{ZGG16}, the notion of an operated Hopf algebra was given, of which the Connes-Kreimer Hopf algebra, equipped with the grafting operator $B^+$, is an example.  
The paper \mcite{ZGG} generalized the operated Hopf algebras to multiple operators case, and defined the $\Omega$-cocycle Hopf algebra.

In various types of operated Hopf algebras, the coproducts and the operators typically satisfy the Hochschild 1-cocycle condition or its generalizations. For example, the Connes-Kreimer Hopf algebra satisfies the Hochschild 1-cocycle condition
\begin{equation} \mlabel{eq:cocyc}
	\Delta P=P\otimes1+(\id\otimes P)\Delta.
	\end{equation}
We would like to emphasize that $\Omega$-cocycle Hopf algebras \mcite{ZGG} satisfy a multiple operators version of this condition.
Left counital Hopf algebras on free commutative Nijenhuis algebras \mcite{ZG,PZGL20} satisfy the variation of the Hochschild 1-cocycle condition, and weighted $\Omega$-cocycle infinitesimal unitary bialgebras defined in \mcite{ZCGL} satisfy weighted cocycle condition. Furthermore, another weighted Hochschild 1-cocycle condition \mcite{ZL25} has been applied in constructing operated algebra structures on free modified Rota-Baxter algebras.

In a recent study~\mcite{MBC23}, the theory of operator Hopf algebras has been applied in Chomsky's Minimalism Program to formalize syntactic merge operations and to study the interaction between semantic computation and syntactic generation with the help of the Hopf algebra structure. This interdisciplinary application not only strengthens the mathematical foundation of generative linguistics but also provides new methodological approaches for natural language processing and computational linguistics.

\subsection{Weighted cocycle conditions and structure of this paper}
In this paper, we generalize the cocycle condition in Eq.~\meqref{eq:cocyc} and its multi-operator version to a weighted cocycle condition, thereby obtaining a new rooted tree Hopf algebra.
Specifically, for a nonempty set $\Omega$ and constants $\lambda_\Omega:=\{\lambda_\omega\,|\,\omega\in\Omega\}$, we consider the {\bf Hochschild 1-cocycle condition with weight $\lambda_\Omega$}
\begin{equation}\mlabel{eq:gcocyc}
\Delta P_\omega=P_\omega\otimes1+(\id\otimes P_\omega)\Delta+\lambda_\omega\id\otimes1, \quad\omega\in\Omega.
\end{equation}
We then apply them to define $\Omega$-cocycle Hopf algebras {\em with weights} and generalize previous studies on $\Omega$-cocycle Hopf algebras to this setting. 

In Section \mref{sec:oophopf}, we first review the basic concepts of planar rooted trees and planar rooted forests, leading to the notions of $(X, \Omega)$-decorated planar rooted trees and forests decorated by a set $X$ and a set $\Omega$, where internal vertices are decorated only by elements from $\Omega$. Then, on the free module $\hrts$ generated by $(X, \Omega)$-decorated rooted forests $\calf_{\rm RT}(X,\Omega)$, we construct a new coproduct satisfying the Hochschild 1-cocycle condition of weight $\lambda_\Omega$ (Eq.~\meqref{eq:cdbp}), extending the Hochschild 1-cocycle condition of weight $\lambda$ \mcite{ZL25} to the case of multiple operators. Generalizing the admissible-cut description of the coproduct in the Connes-Kreimer Hopf algebra, we provide a combinatorial interpretation of the new coproduct $\col$ (Theorem \mref{thm:com}). 

In Section \mref{sec:cocyhopf}, we show that this coproduct equips the algebra $\hrts$ of $(X,\Omega)$-decorated forests with a bialgebra structure (Theorem \mref{thm:bia}). Furthermore, we prove that $\hrts$ is a Hopf algebra (Theorem \mref{thm:hrth}). At the end, we combine this Hopf algebra with operated algebras. We begin by recalling the definition of an $\Omega$-operator algebra. Inspired by the Hochschild 1-cocycle condition of weighted $\lambda_\Omega$, we introduce the concept of $\Omega$-cocycle Hopf algebras with weights, and thereby show that $\hrts$ is a free $\Omega$-cocycle Hopf algebra (Theorem \mref{thm:propm}).

\smallskip  
\noindent
{\bf  Convention.}
Throughout this paper, we fix a unitary commutative ring $\bfk$ as the base ring for all modules, algebras, coalgebras, bialgebras, tensor products, and linear maps. Unless otherwise specified, all algebras mentioned in this work are assumed to be unitary and associative.

\section{A generalization of the construction of Connes-Kreimer Hopf algebras}
\mlabel{sec:oophopf}
In this section, we first recall the concept of the $(X, \Omega)$-decorated rooted trees and forests as defined in \mcite {ZGG16}, where the leaves and internal vertices are decorated by two distinct sets. Subsequently, we define a new coproduct and counit in the free $\bfk$-module $\hrts$ generated by $(X, \Omega)$-decorated rooted forests, which generalizes the construction in \cite{ZGG}.

\subsection{$(X, \Omega)$-decorated rooted forests}
\mlabel{ss:ck}

We recall basic notions of planar rooted trees and furthermore doubly decorated rooted trees. 

A {\bf tree} is an undirected graph in which any two vertices are connected by exactly one path, or equivalently a connected acyclic undirected graph.
A {\bf rooted tree} is a tree in which a particular vertice is distinguished from the others and called the {\bf root}.
A {\bf planar rooted tree} is a rooted tree with a fixed embedding into the plane.

Let $\calt$ denote the set of planar rooted trees and $\calf$ the set of {\bf planar rooted forests}. Algebraically, $\calf$ can be defined as the free monoid generated by $\calt$, denoted by $\calf:=M(\calt)$, whose multiplication is the concatenation product $\mul$ (often omitted for brevity). The {\bf empty tree} in $\calf$ is denoted by $1$, the unit of $\calf$.
A {\bf planar rooted forest} $F\in \calf$ is a concatenation of
planar rooted trees, denoted by $F=T_1 \cdots T_n$, where $T_1,\ldots,T_n \in \calt$.
We will use the convention that $F = 1$ when $n = 0$.

\begin{defn} \mcite{ZGG16,ZGG}
\begin{enumerate}
\item Let $X$ be a set and $\Omega$ be a nonempty set that is disjoint from X. A {\bf $(X, \Omega)$-decorated} rooted tree is a planar rooted tree whose internal vertices (which are not leaves) are decorated only by elements of $\Omega$ and leaf vertices are decorated by elements of $X\sqcup \Omega$. Denote by $\calt_{\rm RT}(X,\Omega)$ (resp. $\calf_{\rm RT}(X,\Omega)$) the set of $(X, \Omega)$-decorated rooted trees (resp. forests). The only vertex of the tree $\sbullet$ is viewed as a leaf vertex.
\item Let $\hck(X,\Omega)$ be the free $\bfk$-module spanned by the $\calf_{\rm RT}(X,\Omega)$:
       \begin{equation*}
          \hrts:=\bfk\calf_{\rm RT}(X,\Omega)=\bfk M(\calt_{\rm RT}(X,\Omega))=\bfk\langle \calt_{\rm RT}(X,\Omega)\rangle.
      \end{equation*}
For any $\omega\in\Omega$, define
        \begin{equation*}
           B^+_\omega : \hrts \to \hrts
           \end{equation*}
           to be the grafting operator by sending $1$ to $\sbullet_\omega\,,\omega\in\Omega$ and sending a rooted forest in $\calf_{\rm RT}(X,\Omega)$ to its grafting with the new root decorated by $\omega\in\Omega$. Then the space $\hrts$ forms an operated unitary algebra under the concatenation product $\mul$.
 \end{enumerate}
 \mlabel{de:oua}
\end{defn}

\begin{exam}
    The first few $(X, \Omega)$-decorated rooted trees are listed below:
    $$\etree,\ \tdun{$\alpha$},\ \, \tdun{$x$},\ \, \tddeux{$\alpha$}{$\beta$},\ \,  \tddeux{$\alpha$}{$x$}, \ \, \tdtroisun{$\alpha$}{$y$}{$x$},\ \,\tdtroisun{$\alpha$}{$\beta$}{$x$}, \ \,\tdtroisun{$\alpha$}{$y$}{$\gamma$},\ \,\tdtroisun{$\alpha$}{$\gamma$}{$\beta$}, \ \, \tdquatretrois{$\alpha$}{$\beta$}{$y$}{$x$},\ \, \tdquatretrois{$\alpha$}{$\beta$}{$\gamma$}{$x$}, \ \, \tdquatretrois{$\alpha$}{$\beta$}{$y$}{$\omega$},\ \, \tdquatretrois{$\alpha$}{$\beta$}{$\gamma$}{$\omega$},\quad x,y\in X,\ \alpha,\beta,\gamma,\omega\in \Omega.$$
    However, the following are not in $\calf_{\rm RT}(X,\Omega)$ since some of the internal vertices are decorated by elements of $X$:
    $$\tddeux{$x$}{$\alpha$},\ \, \tddeux{$x$}{$y$},\ \, \tdtroisdeux{$x$}{$\beta$}{$\alpha$},\ \, \tdquatretrois{$\alpha$}{$x$}{$\gamma$}{$\beta$},\quad x,y\in X,\ \alpha,\beta,\gamma\in \Omega.   $$

\end{exam}

\begin{exam}
    Here are some examples of the grafting operators $B^+ _{\omega}$, $\omega\in\Omega$:
    $$B^+ _{\omega}(1)=\sbullet_{\omega}\,,\ \  B^+ _{\omega}(\sbullet_{x})=\tddeux{$\omega$}{$x$}\,,\ \  B^+ _{\omega}(\tdun{$x$}\tddeux{$\beta$}{$y$})=\tdquatretrois{$\omega$}{$\beta$}{$y$}{$x$}.
    $$
\end{exam}

Let $F=T_1 \cdots T_k \in \calf_{\rm RT}(X,\Omega)$, where $k \geq 1$, $T_i \neq 1$ and $T_i \in \calt_{{\rm RT}}(X,\Omega)$ for $1 \leq i \leq k$. We define $\bre(F) := k$ to be the {\bf breadth} of $F$ with the convention that $\bre(1)=0$. Next, we will define the depth of $(X, \Omega)$-decorated rooted forests by a recursion, yielding a filtered structure on $\calf_{{\rm RT}}(X,\Omega)$.

 Denote $\sbullet_X:=\{\sbullet_x\mid x\in X\}$. Define
 \begin{equation*}
     \calf_{{\rm RT},0}:=M(\sbullet_X)=S(\sbullet_X)\cup\{1\},
 \end{equation*}
 where $M(\sbullet_X)$ (resp. $S(\sbullet_X)$) is the submonoid (resp. subsemigroup) of $\calf_{{\rm RT}}(X,\Omega)$ generated by $\sbullet_X$.
 Assume that $\calf_{{\rm RT},n},n\geq 0$, has been defined. Then define
 \begin{equation}
     \calf_{{\rm RT},n+1}:=M(\sbullet_X \sqcup B^+ _{\Omega}(\calf_{{\rm RT},n}))=M\Big(\sbullet_X \sqcup (\sqcup_{\omega\in\Omega}B^+_{\omega}(\calf_{{\rm RT},n}))\Big).
     \mlabel{eq:frtn}
 \end{equation}
 Through the above construction, we have $\calf_{{\rm RT},n}\subseteq\calf_{{\rm RT},n+1},n\geq0$. Then
\begin{equation*}
    \calf_{{\rm RT}}(X,\Omega)=\dirlim\calf_{{\rm RT},n}=\bigcup_{n=0}^{\infty}\calf_{{\rm RT},n}.
\end{equation*}
 Elements $F\in \calf_{{\rm RT},n} \backslash \calf_{{\rm RT},n-1}$ are said to have {\bf depth} $n$, denoted by $\dep(F) = n$. For $F =
 T_1 \cdots F_k$, we obtain
 \begin{equation*}
     \dep(F) = \text{max}\{\dep(T_i)\mid i = 1,\ldots ,k\}.
 \end{equation*}
Intuitively, $\dep(F)$ is largest number of elements in $\Omega$ in a path from the root of $F$ to one of the leaves. 

We emphasize that $\dep(F)$ is not the usually depth of a rooted tree, defined to be the longest path from the root of $F$ to one of the leaves. Rather, notice that a leaf of $F\in \calf_{{\rm RT}}(X,\Omega)$ can be decorated by either $X$ or $\Omega$. If one of the longest path has its leaf decorated by $\Omega$, then $\dep(F)$ is the length of the longest path plus one. If all the longest paths has their leaves decorated by $X$, then $\dep(F)$ is simply the length of the longest path.

We demonstrate this property by some examples. 
\begin{exam} 
Let $x,y \in X$ and $\alpha,\beta \in \Omega.$ We have
$$\dep(\etree)=\dep(\sbullet_x)=\dep(\sbullet_x\sbullet_y)=0,$$
$$\dep(\sbullet_\alpha)=\dep(\tddeux{$\alpha$}{$x$})=1,$$
$$\dep(\tddeux{$\alpha$}{$\beta$}) 
=\dep(\ \tdquatretrois{$\alpha$}{$\beta$}{$x$}{$y$})=2.$$
\end{exam}

\subsection{Constructions of coproducts and counits on $(X, \Omega)$-decorated rooted forests}
Let $X$ be  a set and $\Omega$ be a nonempty set. We now generalize the cocycle Hopf algebra structure in Eq.~\meqref{eq:cocyc} characterizing the Connes-Kreimer Hopf algebra of rooted trees to an $\Omega$-cocycle Hopf algebra structure on $(X,\Omega)$-decorated rooted forests.

Fix maps 
\begin{equation}
\mu:X\to \bfk, \quad x\mapsto \mu_x, \ x\in X
\mlabel{eq:mumap}
\end{equation}
and
\begin{equation}
 \lambda:\Omega\to \bfk, \quad \omega\mapsto \lambda_\omega, \  \omega\in \Omega.
\mlabel{eq:lambdamap}
\end{equation}
For the new coproduct $\col=\col{}_{\mu,\lambda}$ on $\hrts$, we will define $\col(F)$ for basis elements $F\in \calf_{{\rm RT}}(X,\Omega)$ by induction on the depth, and then extend by linearity.

For the initial step of $\dep(F)=0$, we  define
\begin{equation}
\col(F) :=
\left\{
\begin{array}{ll}
 \etree \ot \etree, & \text{ if } F = \etree, \\
\sbullet_x\ot \etree+\etree \ot \sbullet_x+\mu_x\etree \ot \etree, & \text{ if } F = \sbullet_x \text{ for some } x \in X,\\
\col(\sbullet_{x_1})\cdots\col(\sbullet_{x_m}),
& \text{ if }  F=\sbullet_{x_1}\cdots \sbullet_{x_m} \text{ with } m\geq 2.
\end{array}
\right.
\mlabel{eq:dele}
\end{equation}
For a given $k\geq 0$, assume that $\col(F)$ have been defined for all $F$ with $\dep(F)\leq k$. Consider $F$ with $\dep(F)=k+1$.
First assume that the breadth $\bre(F)$ is one, that is, $F$ is a tree with $\dep(F)\geq 1$.
Then we can write
$F=\pl(\lbar{F})$ for some $\omega\in \Omega$ and $\lbar{F}\in\calf_{{\rm RT}}(X,\Omega)$, and define
\begin{equation}
\col(F)=\col \pl(\lbar{F}) :=\pl(\lbar{F}) \otimes \etree + (\id\otimes \pl)\col(\lbar{F})+ \lambda_\omega \lbar{F} \otimes \etree.
\mlabel{eq:dbp}
\end{equation}
Here $\col(\lbar{F})$ is well-defined by the induction hypothesis.

If $\bre(F) \geq 2$, we write $F=T_{1}T_{2}\cdots T_{m}$ with $m\geq 2$ and $T_1, T_2, \ldots, T_m \in \calt_{{\rm RT}}(X,\Omega)$,
and define
\begin{equation}
\col(F):=\col( T_{1})  \col(T_{2}) \cdots \col(T_{m}),
\mlabel{eq:delee1}
\end{equation}
where $\col( T_{1}),\col(T_{2}), \ldots ,\col(T_{m})$ are given in Eqs.(\mref{eq:dele}) and (\mref{eq:dbp}). This completes the inductive definition of $\col$.

Eq.~\meqref{eq:dbp} means that the coproduct $\col$ satisfies the {\bf Hochschild 1-cocycle condition of weight $\lambda_\Omega$}:
\begin{equation}
	\col \pl= \pl \otimes \etree + (\id\otimes \pl)\col+\lambda_\omega  \id \otimes \etree.
	\mlabel{eq:cdbp}
\end{equation}

\begin{remark}
    \begin{enumerate}
        \item When $\lambda=0$, then Eq.~\meqref{eq:cdbp} reduces to the cocycle condition defined in \mcite{ZGG16}.
        \item When $\Omega$ is a singleton, the Eq.~\meqref{eq:cdbp} becomes the Hochschild 1-cocycle condition of weight $\lambda$ that is introduced in \mcite{ZL25}.
        \item Naturally, when $\lambda=0$ and $\Omega$ is a singleton, we recover the classical Connes-Kreimer cocycle Hopf algebra structure in Eq.~\meqref{eq:cocyc}. 
    \end{enumerate}
\end{remark}

\begin{exam} \mlabel{ex:copex}
    Let $x,y\in X$ and $\alpha,\beta\in\Omega$. Then
\begin{align*}
\col(\tdun{$x$})=&\ \tdun{$x$}\otimes1+1\otimes\tdun{$x$}+\mu_x1\otimes1,\\
\col(\tdun{$\alpha$})
=&\ \tdun{$\alpha$} \otimes1+1\otimes \tdun{$\alpha$}+\lambda_\alpha1\otimes1,\\
\col(\tddeux{$\alpha$}{$x$})
=&\ \tddeux{$\alpha$}{$x$}\otimes1+\tdun{$x$} \otimes \tdun{$\alpha$}+1\otimes\tddeux{$\alpha$}{$x$}+\lambda_\alpha\tdun{$x$} \otimes1+\mu_x1\otimes\tdun{$\alpha$},\\
\col(\tdquatretrois{$\alpha$}{$\beta$}{$x$}{$\gamma$})
=&\ \tdquatretrois{$\alpha$}{$\beta$}{$x$}{$\gamma$}\ot\etree+\tdun{$\gamma$}\ot\tdtroisdeux{$\alpha$}{$\beta$}{$x$}+\mu_x\tdun{$\gamma$}\ot\tddeux{$\alpha$}{$\beta$}
+\tddeux{$\beta$}{$x$}\ot\tddeux{$\alpha$}{$\gamma$}+\lambda_\gamma\tddeux{$\beta$}{$x$}\ot\tdun{$\alpha$}+\tdun{$x$}\ot\tdtroisun{$\alpha$}{$\beta$}{$\gamma$}+\lambda_\beta\tdun{$x$}\ot\tddeux{$\alpha$}{$\gamma$}\\
&+\lambda_\gamma\tdun{$x$}\ot\tddeux{$\alpha$}{$\beta$}+\lambda_\beta\lambda_\gamma\tdun{$x$}\ot\tdun{$\alpha$}+\tdun{$\gamma$}\tddeux{$\beta$}{$x$}\ot\tdun{$\alpha$}+\lambda_\alpha\tdun{$\gamma$}\tddeux{$\beta$}{$x$}\ot\etree+\tdun{$\gamma$}\tdun{$x$}\ot\tddeux{$\alpha$}{$\beta$}+\lambda_\beta\tdun{$\gamma$}\tdun{$x$}\ot\tdun{$\alpha$}\\
&+1\ot\tdquatretrois{$\alpha$}{$\beta$}{$x$}{$\gamma$}+\mu_x 1\ot\tdtroisun{$\alpha$}{$\beta$}{$\gamma$}+\lambda_\gamma 1\ot\tdtroisdeux{$\alpha$}{$\beta$}{$x$}+\lambda_\gamma\mu_x 1\ot\tddeux{$\alpha$}{$\beta$}.
\end{align*}
\end{exam}

Next, We defined a counit $\epl:\hrts\rightarrow\bfk.$ It only needs to define $\epl$ for basis elements $F\in\calf_{{\rm RT}}(X,\Omega)$ by induction on $\dep(F)$, and then extend by linearity. For the initial step of $\dep(F) = 0$, we define
\begin{equation}
\epl(F) :=
\left\{
\begin{array}{ll}
 1_{\bfk}, & \text{ if } F = \etree, \\
-\mu_x1_{\bfk}, & \text{ if } F = \sbullet_x \text{ for some } x \in X,\\
\epl(\sbullet_{x_1})\cdots\epl(\sbullet_{x_m}) &\\
=(-1)^m \mu_{x_1}\cdots\mu_{x_m}\,1_{\bfk},
& \text{ if }  F=\sbullet_{x_1}\cdots \sbullet_{x_m} \text{ with } m\geq 2.
\end{array}
\right.
 \mlabel{eq:dele2}
\end{equation}
For a given $k\geq 0$, assume that $\epl(F)$ has been defined for all $F$ with $\dep(F)\leq k$. Consider $F$ with $\dep(F)=k+1$. When the breadth $\bre(F)=1$, that is, $F$ is a tree with $\dep(F)\geq 1$. 
Then we can write $F = B^+_{\omega}(\lbar{F}),$ for some $\lbar{F}\in\calf_{{\rm RT}}(X,\Omega)$ and $\omega\in\Omega$, and define
\begin{equation}
    \epl(F)=\epl(B^+_{\omega}(\lbar{F})):=-\lambda_\omega \epl(\lbar{F}).
\mlabel{eq:dbp2}
\end{equation}
 If $\bre(F) \geq 2$, we write $F=T_{1}\cdots T_{m}$ with $m \geq 2$ and $T_{1},\ldots,T_{m}\in\calt_{{\rm RT}}(X,\Omega)$, and then define
 \begin{equation}
     \epl(F)=\epl(T_{1})\cdots\epl(T_{m}),
\mlabel{eq:delee2}
 \end{equation}
 where $\epl(T_{1}),\ldots,\epl(T_{m})$ are defined in Eqs.\meqref{eq:dele2} and \meqref{eq:dbp2}.

\subsection{A combinatorial description of \texorpdfstring{$\col$}{col}}
\mlabel{sec:comdes}
In this section, we provide a combinatorial description of the counit $\epl$ and the coproduct $\col$. 

By a direct observation using Eqs.\,\meqref{eq:dele2} and \meqref{eq:dbp2}, we obtain an explicit formula for $\epl(F)$. Let $|F|$ be the degree (number of vertices) of $F\in \calf_{{\rm RT}}(X,\Omega)$. 
Then 
\begin{equation}
	\epl(F)=(-1)^{|F|}\prod_{x\in F} \mu_x \prod_{\omega\in F} 
	\lambda_{\omega}.
	\mlabel{e:counitexp}
\end{equation}
Here the first (resp. second) product runs over vertices of $F$ decorated by $X$ (resp. by $\Omega$) including repetitions.

For an explicit formula of $\col$, we begin by introducing some basic notions.
Define 
\begin{equation}
\begin{aligned}
    \leafset:\sbullet_{X\sqcup\Omega} &\longrightarrow \bfk\sbullet_{X\sqcup\Omega}\oplus\bfk,\\
    \sbullet_x&\mapsto \widetilde{\sbullet_x}:=\sbullet_x+\mu_x,\ \ x\in X,\\    \sbullet_\omega&\mapsto\widetilde{\sbullet_\omega}:=\sbullet_\omega+\lambda_\omega,\ \ \omega\in\Omega.
\mlabel{eq:lvmap}
\end{aligned}
\end{equation}

Next extend $\leafset$ to a linear operator $\leafsp$ on $\bfk \calf_{{\rm RT}}(X,\Omega)$ by applying $\leafset$ to each of the {\it leaf} vertices of a $(X, \Omega)$-decorated rooted forest $F\in \calf_{{\rm RT}}(X,\Omega)$. In other words, define
\begin{equation*}
\leafsp(F):=\widetilde{F},
\end{equation*}
where $\widetilde{F}$ is obtained from $F$ by replacing each leaf $\sbullet_x$ (resp. $\sbullet_\omega$) by $\sbullet_x+\mu_x$ (resp. $\sbullet_\omega+\lambda_\omega$).
We also use the convention that $\leafset(1)=1$.

Here are some examples on the action of $\leafsp$:
\begin{exam}
Let $x,x_1,x_2\in X,\ \alpha,\beta,\gamma\in\Omega$. Then
\begin{align*}
  \widetilde{\tdun{$x$}}&=\leafsp(\tdun{$x$})=\tdun{$x$}+\mu_x,\\ \widetilde{\tdun{$\alpha$}}&=\leafsp(\tdun{$\alpha$})=\tdun{$\alpha$}+\lambda_\alpha,\\
\widetilde{\tdun{$x_1$}\ \tdun{$x_2$}}&=\leafsp(\tdun{$x_1$}\ \tdun{$x_2$}\,)=(\tdun{$x_1$}+\mu_{x_1})(\tdun{$x_2$}+\mu_{x_2})=\tdun{$x_1$}\ \tdun{$x_2$}+\mu_{x_2}\tdun{$x_1$}+\mu_{x_1}\tdun{$x_2$}+\mu_{x_1}\mu_{x_2},\\
\widetilde{\tdun{$\alpha$}\tdun{$\beta$}}&=\leafsp(\tdun{$\alpha$}\tdun{$\beta$})=(\tdun{$\alpha$}+\lambda_\alpha)(\tdun{$\beta$}+\lambda_\beta)=\tdun{$\alpha$}\tdun{$\beta$}+\lambda_\beta\tdun{$\alpha$}+\lambda_\alpha\tdun{$\beta$}+\lambda_\alpha\lambda_\beta,\\
\widetilde{\tdun{$x$}\tdun{$\alpha$}}&=\leafsp(\tdun{$x$}\tdun{$\alpha$})=(\tdun{$x$}+\mu_x)(\tdun{$\alpha$}+\lambda_\alpha)=\tdun{$x$}\tdun{$\alpha$}+\lambda_\alpha\tdun{$x$}+\mu_x\tdun{$\alpha$}+\mu_x\lambda_\alpha,\\
\widetilde{\tdquatretrois{$\alpha$}{$\beta$}{$x$}{$\gamma$}}&=\leafsp(\,\tdquatretrois{$\alpha$}{$\beta$}{$x$}{$\gamma$}{})=\qquad\tddquatretrois{$\alpha$}{$\beta$}{$(\tdun{$x$}+\mu_x)$}{$(\tdun{$\gamma$}+\lambda_\gamma)$}\qquad=\tdquatretrois{$\alpha$}{$\beta$}{$x$}{$\gamma$}+\mu_x \tdtroisun{$\alpha$}{$\beta$}{$\gamma$}+\lambda_\gamma \tdtroisdeux{$\alpha$}{$\beta$}{$x$}+\lambda_\gamma\mu_x \tddeux{$\alpha$}{$\beta$}.
  \end{align*}
\end{exam}

As in \mcite{Guo,ZGG}, a {\bf subforest} $F'$ of a $(X, \Omega)$-decorated rooted tree $T$ or, more generally, a rooted forest $F$ consists of a set of vertices of $F$ together with their descendants and edges connecting all these vertices. Let $\calf_{F}$ be the set of subforests of $F$, including the empty tree 1 and the full forest $F$. The quotient $F/F'$ is obtained by removing the subforest $F'$ and the edges connecting the subforest to the rest of $F$. By convention, if $F'=F$, then $F/F'=\etree$, and if $F'=\etree$, then $F/F'=F$.

Then the coproduct of Connes-Kreimer Hopf algebra of rooted trees is given by
\begin{align*}
    \Delta_{CK}(F)=\sum_{G\in\calf_F}G\ot(F/G),
\end{align*}
where $\calf_F$ is the set of subforests of a rooted forest $F$.
The coproduct $\col$ has a similar formula. 
\begin{theorem}
\mlabel{thm:com}
For every $(X, \Omega)$-decorated rooted forest $F\in\calf_{{\rm RT}}(X,\Omega)$, we have
\begin{equation}
    \col(F)=\sum_{G\in\calf_F}G\ot\widetilde{F/G}.
\mlabel{eq:ccd}
\end{equation}
\end{theorem}
\begin{proof}  
To prove Eq.~\meqref{eq:ccd}, we will proceed by induction on $\dep(F)\geq0$. For the initial step of $\dep(F)=0$, we have $F=\etree$ or $\sbullet_{x_1}\cdots\sbullet_{x_m}$ for $x_i\in X$, $i=1,\ldots,m$ and $m\geq 1$. If $F=\etree$, then Eq.~\meqref{eq:ccd} follows from Eq.~\meqref{eq:dele}. If $F=\sbullet_{x_1}\cdots\sbullet_{x_m}$, then for $1\leq i\leq m$, by Eq.~\meqref{eq:dele} we have
$$
\col(\sbullet_{x_i})
=\sbullet_{x_i}\ot\etree+\etree\ot(\sbullet_{x_i}+\mu_{x_i}) \\
=\sbullet_{x_i}\ot\etree+\etree\ot\widetilde{\sbullet_{x_i}} \\
=\sum_{G\in\calf_{\sbullet_{x_i}}}G\ot\widetilde{\sbullet_{x_i}/G}.$$
Note that every subforest of a forest $T_1\cdots T_m$ is of the form $G_1\cdots G_m$ for subforests $G_i$ of $T_i$. Thus, 
\begin{align*}
\col(F)
&=\col(\sbullet_{x_1})\cdots\col(\sbullet_{x_m}) \quad\text{(by Eq.~\meqref{eq:dele})}\\
&=\prod_{i=1}^m \bigg(\sum_{G_i\in\calf_{\sbullet_{x_i}}}G_i\ot\widetilde{\sbullet_{x_i}/G_i} \bigg)\\
&=\sum_{G_1\in\calf_{\sbullet_{x_1}},\ldots,G_m\in\calf_{\sbullet_{x_m}}}\bigg(\prod_{i=1}^m G_i\ot\prod_{i=1}^m\widetilde{\sbullet_{x_i}/G_i} \bigg)\\
&=\sum_{G\in\calf_{\sbullet_{x_1}\cdots\sbullet_{x_m}}}G\ot\widetilde{(\sbullet_{x_1}\cdots\sbullet_{x_m})/G}.
\end{align*}

For the inductive step, fix a $k\geq 0$ and  assume that Eq.~\meqref{eq:ccd} holds when $0\leq\dep(F)\leq k$. Consider the case when $\dep(F)=k+1\geq 1$. If $F=T$ is a rooted tree, then we can write $T=\pl(\lbar{F})$, for some $\lbar{F}\in\calf_{{\rm RT}}(X,\Omega)$ and $\omega\in\Omega$, where $\lbar{F}$ has depth $k$. Note that if a subforest $G$ of $T$ is not $T$ itself, then it is a subforest of $\lbar{F}$ and $T/G=\pl(\lbar{F}/G)$. Then we can apply the induction hypothesis and obtain 
\begin{align*}
	\col(T)=&\col \pl(\lbar{F})\\
	=&\pl(\lbar{F}) \otimes \etree+ \lambda_\omega \lbar{F} \otimes \etree + (\id\otimes \pl)\col(\lbar{F})\quad\text{(by Eq.~\meqref{eq:dbp})}\\
	=&T\ot\etree+\lambda_\omega \lbar{F} \otimes \etree +(\id\otimes \pl)\bigg(\sum_{ G\in\calf_{\lbar{F}}}G\ot\widetilde{\lbar{F}/G}\bigg) \quad\text{(by induction hypothesis)}\\
	=&T\ot\etree+\lambda_\omega \lbar{F} \otimes \etree +\bigg(\sum_{G\in\calf_{\lbar{F}}}G\ot\pl(\widetilde{\lbar{F}/G})\bigg)\\
	=&T\ot\etree+\sum_{G\in\calf_{\lbar{F}}}G\ot\widetilde{T/G}\\
	=&\sum_{G\in\calf_T}G\ot\widetilde{T/G}.
\end{align*}

Finally let $F=T_1\cdots T_n$ with rooted trees $T_1,\ldots,T_n$, $n\geq 2$. Then applying the previous step,  
\begin{align*}
    \col(F)&=\col(T_1\cdots T_n)\\
    &=\col(T_1)\cdots\col(T_n)\quad\text{(by Eq.~\meqref{eq:delee1})}\\
    &=\prod_{i=1}^n\bigg(\sum_{G_i\in\calf_{T_i}}G_i\ot\widetilde{T_i/G_i} \bigg)\\ 
    &=\sum_{G_1\in\calf_{T_1},\ldots,G_n\in\calf_{T_n}}\bigg(\prod_{i=1}^nG_i\ot\prod_{i=1}^n\widetilde{T_i/G_i} \bigg)\\
    &=\sum_{G\in\calf_{T_1\cdots T_n}}G\ot\widetilde{T_1\cdots T_n/G}.
\end{align*}

This completes the proof of Eq.~\meqref{eq:ccd}.
  \end{proof}

\begin{exam} We give an example to compare the Connes-Kreimer coproduct and our coproduct. 
    Let $T=\tdquatretrois{$\alpha$}{$\beta$}{$x$}{$\gamma$}\in\calt_{{\rm RT}}(X,\Omega)$, we have\vspace{.2cm} \\
\begin{tabular}{|c|c|c|c|c|c|c|c|}
\hline
\makecell{sub-\\forest\\ G} & \tdquatretrois{$\alpha$}{$\beta$}{$x$}{$\gamma$} & \tdun{$\gamma$} & \tdun{$x$} & \tddeux{$\beta$}{$x$} & \tdun{$\gamma$}\tdun{$x$} & \tdun{$\gamma$}\tddeux{$\beta$}{$x$} & \etree \\
\hline
T/G & \etree & \tdtroisdeux{$\alpha$}{$\beta$}{$x$} & \tdtroisun{$\alpha$}{$\beta$}{$\gamma$} & \tddeux{$\alpha$}{$\gamma$} & \tddeux{$\alpha$}{$\beta$} & \tdun{$\alpha$} & \tdquatretrois{$\alpha$}{$\beta$}{$x$}{$\gamma$} \\
\hline
$\widetilde{T/G}$ & $\etree$ & $\tdtroisdeux{$\alpha$}{$\beta$}{$x$}+\mu_x\tddeux{$\alpha$}{$\beta$}$ & $\makecell{\tdtroisun{$\alpha$}{$\beta$}{$\gamma$}+\lambda_\beta\tddeux{$\alpha$}{$\gamma$}\\+\lambda_\gamma\tddeux{$\alpha$}{$\beta$}+\lambda_\beta\lambda_\gamma\tdun{$\alpha$}}$ & $\tddeux{$\alpha$}{$\gamma$}+\lambda_\gamma\tdun{$\alpha$}$ & $\tddeux{$\alpha$}{$\beta$}+\lambda_\beta\tdun{$\alpha$}$ & $\tdun{$\alpha$}+\lambda_\alpha$ & $\makecell{\tdquatretrois{$\alpha$}{$\beta$}{$x$}{$\gamma$}+\mu_x \tdtroisun{$\alpha$}{$\beta$}{$\gamma$}\\+\lambda_\gamma \tdtroisdeux{$\alpha$}{$\beta$}{$x$}+\lambda_\gamma\mu_x \tddeux{$\alpha$}{$\beta$}}$ \\
\hline
\end{tabular}
\smallskip

Then
$$ \Delta_{CK}(T)= \sum_{G\in \calf_T} G\otimes T/G$$
as $F$ runs through the above table, and 
$$ \Delta_{{\rm RT}}(T)=\sum_{G\in \calf_T} G\otimes \widetilde{T/G}.$$
After expanding $\widetilde{T/G}$ as in the table, this coproduct coincides with the one in Example\,\mref{ex:copex}. 
\end{exam}

\section{\texorpdfstring{$\Omega$}{Omega}-cocycle Hopf algebra of $(X, \Omega)$-decorated rooted forests}
\mlabel{sec:cocyhopf}

In this section we establish the $\Omega$-cocycle Hopf algebra structure on $\hrts$. We start with the bialgebra property in Section~\mref{ss:bialg}. Then a connected filtered structure is introduced in Section~\mref{ss:hopf} which provides the antipode. In Section~\mref{ss:free}, this Hopf algebra is shown to satisfy a universal property among the newly introduced $\Omega$-cocycle Hopf algebras. 

\subsection{A bialgebra structure on $(X, \Omega)$-decorated rooted trees}
\mlabel{ss:bialg}
Given a positive integer $m$, denote by $[m]$ the set $\{1,2,\ldots,m\}$. 
For a subset $I=\{i_1,\ldots,i_k\}\subseteq [m]$, denote $\sbullet{x_{I}}:=\sbullet_{x_{i_{1}}}\sbullet_{x_{i_{2}}}\cdots\sbullet_{x_{i_{k}}}$.
With the convention that $\sbullet_{x_\emptyset}=1$ 
and denote $\mu_{x_I}:=\mu_{x_{i_{1}}}\mu_{x_{i_{2}}}\cdots\mu_{x_{i_{k}}}$ with the convention that $\mu_{x_\emptyset}=1$.

For later application, we give another formulation of Eq.~\meqref{eq:ccd} in a special case. 
 \begin{lemma}\mlabel{lem:col}
Let $x_{1},\ldots,x_{m}\in X$, and let $\sbullet_{x_1}\cdots\sbullet{x_m}\in \hrts$ with $\mathrm{m}\geq1$. Let $\col$ be as above. Then
$$\col(\sbullet_{x_1}\cdots\sbullet{x_m})=\sum_{I\sqcup J=[m]}\mu_{x_I}\bigg(\sum_{K\sqcup L=J}\sbullet_{x_{K}}\,\otimes\sbullet_{x_{L}}\,\bigg).$$
 \end{lemma}

\begin{proof} 
Applying Eq.~\meqref{eq:ccd} to $F=\sbullet_{x_1}\cdots\sbullet{x_m}$ gives 
\begin{align*}
\col(F) &=\sum_{G\in\calf_{\sbullet_{x_1}\cdots\sbullet_{x_m}}}
G\ot\widetilde{(\sbullet_{x_1}\cdots\sbullet_{x_m})/G}	\\
&=\sum_{U\sqcup V=[m]} \sbullet_{x_U} \ot \widetilde{\sbullet_{x_V}}	\\
&=\sum_{U\sqcup V=[m]} \sbullet_{x_U} \ot \prod_{i\in V} (\sbullet_{x_i}+\mu_{x_i})\\
&= \sum_{U\sqcup V=[m]} \sbullet_{x_U} \ot \sum_{W\sqcup Z=V} \mu_{x_Z} \sbullet_{x_W}\\
&=\sum_{U\sqcup W\sqcup Z=[m]} \mu_{x_Z} \sbullet_{x_U}\ot\sbullet_{x_W}.
\end{align*} 
Then taking $I=Z, J=U\sqcup W, K=U, L=W$ yields the lemma. 
\end{proof}

We first check the compatibility of the coproduct on $\hrts$ with the product. 

\begin{lemma}\mlabel{lem:colh}
 The linear map $\col:\hrts\rightarrow \hrts\ot\hrts$ is an algebra homomorphism.
\end{lemma}
\begin{proof}
We only need to prove
\begin{align}
\mlabel{eq:chom}
\col(FF')=\col(F)\col(F'),\quad \text{for all}\   F,F'\in\calf_{{\rm RT}}(X,\Omega).
\end{align}
If $F=1$ or $F'=1$, without loss of generality, we take $F'=1$. Then by Eq.~\meqref{eq:dele}, we have $\col(F')=1\ot1$, and so $$\col(FF')=\col(F)=\col(F)(1\ot1)=\col(F)\col(F').$$

If $F\neq 1$ and $F'\neq 1$, write
$F=T_{1}\cdots T_{m}$ and $F'=T'_{1}\cdots T'_{m'}$, where $m,m'\geq1$ and $T_{i},T'_{j}\in\calt_{{\rm RT}}(X,\Omega)$ for $1\leq i\leq m$ and $1\leq j\leq m'$. By Eq.~(\mref{eq:delee1}) we have
\begin{align*}
\col(FF')&=\col(T_{1}\cdots T_{m}T'_{1}\cdots T'_{m'})\\
&=\col(T_{1})\cdots\col(T_{m})\col(T'_{1})\cdots\col(T'_{m'})\\
&=\col(T_{1}\cdots T_{m})\col(T'_{1}\cdots T'_{m'})\\
&=\col(F)\col(F'),
\end{align*}
as desired.
\end{proof}

\begin{lemma}\mlabel{lem:eplh}
The linear map $\epl:\hrts\rightarrow\bfk$ is an algebra homomorphism.
\end{lemma}

\begin{proof}
We just need to show
\begin{align}
\mlabel{eq:ehom}
\epl(FF')=\epl(F)\epl(F'),\quad \text{for all}\ F,F'\in\calf_{{\rm RT}}(X,\Omega).
\end{align}
This follows from an argument similar to the one for Lemma~\mref{lem:colh}. 
\end{proof}

We now check the coassociativity of $\col$, beginning with a special case. 

\begin{lemma}\mlabel{lem:coass1}
Let $F=\sbullet x_{1}\cdots\sbullet x_{m}\in\hrts, m\geq1$. Then
\begin{equation}
    (\col\ot\id)\col(F)=(\id\ot\col)\col(F).
\mlabel{eq:coasse1}
\end{equation}
\begin{proof}
We apply induction on $m=\bre(F) \geq1$ to prove Eq.~\meqref{eq:coasse1}. If $m = 1$, then $F = \sbullet_{x_{1}}$
for some $x_{1} \in X$. Thus
\begin{align*}
&(\col\ot\id)\col(F)\\
&=(\col\ot\id)\col(\sbullet_{x_{1}})\\
&=(\col\ot\id)(\sbullet_{x_{1}}\ot1+1\ot\sbullet_{x_{1}}+\mu_{x_1}1\ot1)\\
&=\col(\sbullet_{x_{1}})\ot1+\col(1)\ot\sbullet_{x_{1}}+\mu_{x_1}\col(1)\ot1\\
&=\sbullet_{x_{1}}\ot1\ot1+1\ot\sbullet_{x_{1}}\ot1+\mu_{x_1}1\ot1\ot1+1\ot1\ot\sbullet_{x_{1}}+\mu_{x_1}1\ot1\ot1\\
&=(\id\ot\col)\col(F).
\end{align*}
Assume that Eq.~\meqref{eq:coasse1} is true for $1\leq m\leq i$, for a given $i\geq 1$. Consider the case when $m=i+1\geq 2$. Then we have $F=F'\sbullet_{x_{i+1}}$, where $F'=\sbullet_{x_{1}}\cdots\sbullet_{x_{i}}$ and $\bre(F')=i$.
By Lemma ~\mref{lem:col}, we have
$$\col(F')=\sum_{I\sqcup J=[i]}\mu_{x_I}\bigg(\sum_{K\sqcup L=J}\sbullet_{x_{K}}\otimes\sbullet_{x_{L}}\bigg).$$
By the induction hypothesis, we obtain $(\col\ot\id)\col(F')=(\id\ot\col)\col(F')$, that is
$$\sum_{I\sqcup J=[i]}\mu_{x_I}\bigg(\sum_{K\sqcup L=J}\col(\sbullet_{x_{K}})\otimes\sbullet_{x_{L}}\bigg)=\sum_{I\sqcup J=[i]}\mu_{x_I}\bigg(\sum_{K\sqcup L=J}\sbullet_{x_{K}}\otimes\col(\sbullet_{x_{L}})\bigg). $$
Then
\begin{align*}
&(\col\ot\id)\col(F)\\
&=(\col\ot\id)(\col(F')\col(\sbullet_{x_{i+1}}))\\
&=(\col\ot\id)\left(\bigg(\sum_{I\sqcup J=[i]}\mu_{x_I}\big(\sum_{K\sqcup L=J}\sbullet_{x_{K}}\otimes\sbullet_{x_{L}}\big)\bigg)(\sbullet_{x_{i+1}}\ot \etree+\etree \ot \sbullet_{x_{i+1}}+\mu_{x_{i+1}}\etree \ot \etree)\right)\\
&=(\col\ot\id)\left(\sum_{I\sqcup J=[i]}\mu_{x_I}\bigg(\sum_{K\sqcup L=J}\sbullet_{x_K}\sbullet_{x_{i+1}}\ot\sbullet_{x_L}+\sbullet_{x_K}\ot\sbullet_{x_L}\sbullet_{x_{i+1}}+\mu_{x_{i+1}}\sbullet_{x_{K}}\otimes\sbullet_{x_{L}}\bigg)\right)\\
&=\sum_{I\sqcup J=[i]}\mu_{x_I}\bigg(\sum_{K\sqcup L=J}\col(\sbullet_{x_K})\col(\sbullet_{x_{i+1}})\ot\sbullet_{x_L}+\col(\sbullet_{x_K})\ot\sbullet_{x_L}\sbullet_{x_{i+1}}+\mu_{x_{i+1}}\col(\sbullet_{x_{K}})\otimes\sbullet_{x_{L}}\bigg)\\
&=\left(\sum_{I\sqcup J=[i]}\mu_{x_I}\bigg(\sum_{K\sqcup L=J}\col(\sbullet_{x_{K}})\otimes\sbullet_{x_{L}}\bigg)\right)(\col(\sbullet_{x_{i+1}})\ot\etree+\etree\ot\etree\ot\sbullet_{x_{i+1}}+\mu_{x_{i+1}}\etree\ot\etree\ot\etree)\\
&=\left(\sum_{I\sqcup J=[i]}\mu_{x_I}\bigg(\sum_{K\sqcup L=J}\sbullet_{x_{K}}\otimes\col(\sbullet_{x_{L}})\bigg)\right)(\sbullet_{x_{i+1}}\ot\etree\ot\etree+\etree\ot\col(\sbullet_{x_{i+1}})+\mu_{x_{i+1}}\etree\ot\etree\ot\etree)\\
&\hspace{8cm}\text{(by induction hypothesis)}\\
&=\sum_{I\sqcup J=[i]}\mu_{x_I}\bigg(\sum_{K\sqcup L=J}\sbullet_{x_K}\sbullet_{x_{i+1}}\ot\col(\sbullet_{x_L})+\sbullet_{x_K}\ot\col(\sbullet_{x_L})\col(\sbullet_{x_{i+1}})+\mu_{x_{i+1}}\sbullet_{x_{K}}\otimes\col(\sbullet_{x_{L}})\bigg)\\
&=(\id\ot\col)\col(F).
\end{align*}
This completes the inductive proof.
\end{proof}
\end{lemma}

For the coassociativity in general, we have 
\begin{lemma}\mlabel{lem:coass2}
Let $F\in\hrts$. Then
\begin{equation}
 (\col\ot\id)\col(F)=(\id\ot\col)\col(F).
 \mlabel{eq:coasse2}
\end{equation}
\begin{proof}
By the linearity, we shall use induction on $\dep(F)$ to prove Eq.~\meqref{eq:coasse2}, where $F\in\calf_{{\rm RT}}(X,\Omega)$. If $\dep(F) = 0$, then $F = 1$ or $F = \sbullet_{x_{1}}\cdots\sbullet_{x_{m}}$ for $m \geq1$ and $x_{i} \in X$. If $F = 1$, we have
$$(\col\ot\id)\col(1)=(\col\ot\id)(1\ot1)=1\ot1\ot1=(\id\ot\col)\col(1).$$
If $F = \sbullet_{x_{1}}\cdots\sbullet_{x_{m}}$, then Eq.~\meqref{eq:coasse2} holds by Lemma ~\mref{lem:coass1}. Let $j\geq0$. Assume that Eq.~\meqref{eq:coasse2} holds for $0 \leq \dep(F) \leq j$ and consider the case when $\dep(F) = j + 1$. We next prove Eq.~\meqref{eq:coasse2} by induction on $\bre(F) \geq 0$. Since $\dep(F) = j + 1 \geq 1$, we have $F\neq 1$ and so $\bre(F) \geq 1$. If $\bre(F) = 1$, we can write $F = B^+_{\omega}(\lbar{F})$ for some $\lbar{F} \in \calf_{{\rm RT}}(X,\Omega)$ and $\omega\in\Omega$. Hence
\begin{align*}
&(\col\ot\id)\col(F)\\
&=(\col\ot\id)\col(B^+_{\omega}(\lbar{F}))\\
&=(\col\ot\id)(B^+_{\omega}(\lbar{F})\ot1+\lambda_\omega\lbar{F}\ot1+(\id\ot B^+_{\omega})\col(\lbar{F}))\quad \text{(by Eq.~\meqref{eq:dbp})}\\
&=\col(B^+_{\omega}(\lbar{F}))\ot1+\lambda_\omega\col(\lbar{F})\ot1+(\col\ot B^+_{\omega})\col(\lbar{F})\\
&=B^+_{\omega}(\lbar{F})\ot1\ot1+\lambda_\omega\lbar{F}\ot1\ot1+(\id\ot B^+_{\omega})\col(\lbar{F})\ot1+\lambda_\omega\col(\lbar{F})\ot1\\
&\quad+(\id\ot\id\ot B^+_{\omega})(\col\ot\id)\col(\lbar{F})\quad \text{(by Eq.~\meqref{eq:dbp})}\\
&=B^+_{\omega}(\lbar{F})\ot1\ot1+\lambda_\omega\lbar{F}\ot1\ot1+(\id\ot B^+_{\omega})\col(\lbar{F})\ot1+\lambda_\omega\col(\lbar{F})\ot1\\
&\quad+(\id\ot\id\ot B^+_{\omega})(\id\ot\col)\col(\lbar{F})\quad\text{(by the induction hypothesis on depth)}\\
&=B^+_{\omega}(\lbar{F})\ot\col(1)+\lambda_\omega\lbar{F}\ot\col(1)+(\id\ot\col B^+_{\omega})\col(\lbar{F})\quad \text{(by Eq.~\meqref{eq:cdbp}}\\
&=(\id\ot\col)(B^+_{\omega}(\lbar{F})\ot1+\lambda_\omega\lbar{F}\ot1+(\id\ot B^+_{\omega})\col(\lbar{F}))\\
&=(\id\ot\col)\col(B^+_{\omega}(\lbar{F}))\\
&=(\id\ot\col)\col(F).
\end{align*}
Let $i\geq1$. Assume that Eq.~\meqref{eq:coasse2} holds for $\dep(F) = j +1$ and $1 \leq \bre(F) \leq i$. Next consider the case when $\dep(F) = j+1$ and $\bre(F) = i+1 \geq 2$. Then $F =F'F''$ for some $F' ,F'' \in \calf_{{\rm RT}}(X,\Omega)$ with $1 \leq \bre(F'),\bre(F'') \leq i$. Write
 $$\col(F')=\sum_{(F')}F'_{(1)}\ot F'_{(2)}\quad\text{and}\quad\col(F'')=\sum_{(F'')}F''_{(1)}\ot F''_{(2)}.$$
By the induction hypothesis on breadth, we have
\begin{align*}
(\col\ot\id)\col(F')&=(\id\ot\col)\col(F'),\\
(\col\ot\id)\col(F'')&=(\id\ot\col)\col(F'').
\end{align*}
That is
\begin{equation}
\begin{aligned}\mlabel{eq:coaee}
\sum_{(F')}\col(F'_{(1)})\ot F'_{(2)}&=\sum_{(F')}F'_{(1)}\ot \col(F'_{(2)}),\\
\sum_{(F'')}\col(F''_{(1)})\ot F''_{(2)}&=\sum_{(F'')}F''_{(1)}\ot \col(F''_{(2)}).
\end{aligned}
\end{equation}
Thus
\begin{align*}
(\col\ot\id)\col(F)&=(\col\ot\id)(\col(F')\col(F''))\\
&=(\col\ot\id)\bigg(\sum_{F'}\sum_{F''}F'_{(1)}F''_{(1)}\ot F'_{(2)}F''_{(2)}\bigg)\\
&=\sum_{F'}\sum_{F''}\col(F'_{(1)}F''_{(1)})\ot F'_{(2)}F''_{(2)}\\
&=\sum_{F'}\sum_{F''}\col(F'_{(1)})\col(F''_{(1)})\ot F'_{(2)}F''_{(2)}\\
&=(\sum_{(F')}\col(F'_{(1)})\ot F'_{(2)})(\sum_{(F'')}\col(F''_{(1)})\ot F''_{(2)}).
\end{align*}
 Similarly, we have
 $$(\id\ot\col)\col(F)=\bigg(\sum_{(F')}F'_{(1)}\ot \col(F'_{(2)})\bigg)\bigg(\sum_{(F'')}F''_{(1)}\ot \col(F''_{(2)})\bigg).$$
Then Eq.~\meqref{eq:coasse2} follows from Eq.~\meqref{eq:coaee}.
 \end{proof}
\end{lemma}

We verify the counit property in the next two lemmas. 
\begin{lemma}\mlabel{lem:conui1}
Let $F=\sbullet x_{1}\cdots\sbullet x_{m}\in\hrts,m\geq1$. Then
\begin{equation}\mlabel{eq:conui1}
(\epl\ot\id)\col(F)=\beta_{l}(F)\quad\text{and}\quad (\id\ot\epl)\col(F)=\beta_{r}(F),
\end{equation}
where $\beta_{l}:\hrts\rightarrow\bfk\ot\hrts$ is given by $F\rightarrow1_{k}\ot F$ and $\beta_{r}:\hrts\rightarrow\hrts\ot \bfk$ is given by $F\rightarrow F\ot 1_{k}$.
\end{lemma}

\begin{proof}
We will prove Eq.~\meqref{eq:conui1} by induction on $m=\bre(F)\geq1$. For the initial step of $m=1$, we have $F=\sbullet_{x_{1}}$ for some $x_{1}\in X$. Then
\begin{align*}
(\epl\ot\id)\col(F)
&=(\epl\ot\id)(\sbullet_{x_{1}}\ot1+1\ot\sbullet_{x_{1}}+\mu_{x_1}1\ot1)\\
&=\epl(\sbullet_{x_{1}})\ot1+1_{k}\ot\sbullet_{x_{1}}+\mu_{x_1}1_{k}\ot1\\
&=-\mu_{x_1}1_{k}\ot1+1_{k}\ot\sbullet_{x_{1}}+\mu_{x_1}1_{k}\ot1\\
&=1_{k}\ot\sbullet_{x_{1}}\\
&=\beta_{l}(F),
\end{align*}
and similarly,
$$(\id\ot\epl)\col(F)=\beta_r(F).$$

Let $i\geq1$. Assume that Eq.~\meqref{eq:conui1} holds for $1\leq m\leq i$ and consider the case when $m=i+1\geq2$. We can write $F=F'\sbullet_{x_{i+1}}$, where $F'=\sbullet_{x_{1}}\cdots\sbullet_{x_{i}}$ and $\bre(F')=i$. By Lemma ~\mref{lem:col}, we have
$$\col(F')=\sum_{I\sqcup J=[i]}\mu_{x_I}\bigg(\sum_{K\sqcup L=J}\sbullet_{x_{K}}\otimes\sbullet_{x_{L}}\bigg).$$
By the induction hypothesis, we obtain $(\epl\ot\id)\col(F')=\beta_{l}(F'),\  (\id\ot\epl)\col(F')=\beta_{r}(F')$, that is
\begin{align*}
 \sum_{I\sqcup J=[i]}\mu_{x_I}\bigg(\sum_{K\sqcup L=J}\epl(\sbullet_{x_{K}})\otimes\sbullet_{x_{L}}\bigg)=1_{k}\ot F' ,
\end{align*}
\begin{align*}
\sum_{I\sqcup J=[i]}\mu_{x_I}\bigg(\sum_{K\sqcup L=J}\sbullet_{x_{K}}\otimes\epl(\sbullet_{x_{L}})\bigg)=F'\ot1_{k} .
\end{align*}
Thus
\begin{align*}
&(\epl\ot\id)\col(F)\\
&=(\epl\ot\id)(\col(F')\col(\sbullet_{x_{k}}))\\
&=(\epl\ot\id)\left(\bigg(\sum_{I\sqcup J=[i]}\mu_{x_I}\big(\sum_{K\sqcup L=J}\sbullet_{x_{K}}\otimes\sbullet_{x_{L}}\big)\bigg)(\sbullet_{x_{i+1}}\ot \etree+\etree \ot \sbullet_{x_{i+1}}+\mu_{x_{i+1}}\etree \ot \etree)\right)\\
&=(\epl\ot\id)\left(\sum_{I\sqcup J=[i]}\mu_{x_I}\bigg(\sum_{K\sqcup L=J}\sbullet_{x_K}\sbullet_{x_{i+1}}\ot\sbullet_{x_L}+\sbullet_{x_K}\ot\sbullet_{x_L}\sbullet_{x_{i+1}}+\mu_{x_{i+1}}\sbullet_{x_{K}}\otimes\sbullet_{x_{L}}\bigg)\right)\\
&=\sum_{I\sqcup J=[i]}\mu_{x_I}\bigg(\sum_{K\sqcup L=J}\epl(\sbullet_{x_K})\epl(\sbullet_{x_{i+1}})\ot\sbullet_{x_L}+\epl(\sbullet_{x_K})\ot\sbullet_{x_L}\sbullet_{x_{i+1}}+\mu_{x_{i+1}}\epl(\sbullet_{x_{K}})\otimes\sbullet_{x_{L}}\bigg)\\
&=\left(\sum_{I\sqcup J=[i]}\mu_{x_I}\bigg(\sum_{K\sqcup L=J}\epl(\sbullet_{x_{K}})\otimes\sbullet_{x_{L}}\bigg)\right)(\epl(\sbullet_{x_{i+1}})\ot\etree+1_k\ot\sbullet_{x_{i+1}}+\mu_{x_{i+1}}1_k\ot1)\\
&=(1_k\ot F')(\epl(\sbullet_{x_{i+1}})\ot\etree+1_k\ot\sbullet_{x_{i+1}}+\mu_{x_{i+1}}1_k\ot1)\quad\text{(by induction hypothesis)}\\
&=\epl(\sbullet_{x_{i+1}})\ot F'+1_k\ot F'\sbullet_{x_{i+1}}+\mu_{x_{i+1}}1_k\ot F'\\
&=-\mu_{x_{i+1}}1_k\ot F'+1_k\ot F+\mu_{x_{i+1}}1_k\ot F'\\
&=1_{k}\ot F\\
&=\beta_{l}(F).
\end{align*}
A similarly computation gives
$$(\id\ot\epl)\col(F) = \beta_r(F).$$
This completes the induction on the breadth.
\end{proof}

\begin{lemma}\mlabel{lem:coal}
Let $F\in\hrts$. Then
\begin{equation}\mlabel{eq:coal}
(\epl\ot\id)\col(F)=\beta_{l}(F)\quad\text{and}\quad (\id\ot\epl)\col(F)=\beta_{r}(F).
\end{equation}
\end{lemma}
\begin{proof}
By the linearity, we use induction on $\dep(F) \geq 0$ to prove Eq.~\meqref{eq:coal} for $F\in \calf_{{\rm RT}}(X,\Omega)$. 
We just prove the first identity since the prove of the second one is the same. 

For the initial step of $\dep(F) = 0$, we have $F = 1$ or $F =\sbullet_{x_{1}}\cdots\sbullet_{x_{m}},\ m\geq1,x_{i}\in X$. If $F = 1$, then
\begin{align*}
(\epl\ot\id)\col(1)=(\epl\ot\id)(1\ot1)=\epl(1)\ot1=1_{k}\ot1=\beta_{l}(1).
\end{align*}
If $F =\sbullet_{x_{1}}\cdots\sbullet_{x_{m}}$, then Eq.~\meqref{eq:coal} holds by Lemma ~\mref{lem:conui1}. Let $j\geq0$. Assume that Eq.~\meqref{eq:coal} holds when $0 \leq \dep(F) \leq j$ and consider the case when $\dep(F) = j+1\geq 1$. We next reduce to prove Eq.~\meqref{eq:coal} by induction on breadth. Since $\dep(F) \geq 1$, we have $F\neq 1$ and so $\bre(F) \geq 1$. If $\bre(F) = 1$, we can write $F = B^+_{\omega}(\lbar{F})$ for some $F \in F_{{\rm RT}}(X,\Omega)$ and $\omega\in\Omega$. Then
\begin{align*}
&(\epl\ot\id)\col(F)\\
&=(\epl\ot\id)\col(B^+_{\omega}(\lbar{F}))\\
&=(\epl\ot\id)(B^+_{\omega}(\lbar{F})\ot1+\lambda_\omega\lbar{F}\ot1+(\id\ot B^+_{\omega})\col(\lbar{F}))\quad\text{(by Eq.~\meqref{eq:dbp})}\\
&=\epl(B^+_{\omega}(\lbar{F}))\ot1+\lambda_\omega\epl(\lbar{F})\ot1+(\epl\ot B^+_{\omega})\col(\lbar{F})\\
&=-\lambda_\omega\epl(\lbar{F})\ot1+\lambda_\omega\epl(\lbar{F})\ot1+(\id\ot B^+_{\omega})(\epl\ot\id)\col(\lbar{F})\quad\text{(by Eq.~\meqref{eq:dbp2})}\\
&=(\id\ot B^+_{\omega})\beta_{l}(\lbar{F})\quad\text{(by the induction hypothesis on depth)}\\
&=(\id\ot B^+_{\omega})(1_{k}\ot\lbar{F})\\
&=1_{k}\ot B^+_{\omega}(\lbar{F})\\
&=1_{k}\ot F\\
&=\beta_{l}(F).
\end{align*}

Let $i\geq1$. Assume that the first identity in Eq.~\meqref{eq:coal} holds when $1 \leq \bre(F) \leq i$. Consider the case when $\bre(F) = i+1 \geq 2$. So we can write $F = F'F''$ , where $1 \leq\bre(F'),\bre(F'') \leq i$. Write
\begin{align*}
\col(F')=\sum_{(F')}F'_{(1)}\ot F'_{(2)}\quad\text{and}\quad \col(F'')=\sum_{(F'')}F''_{(1)}\ot F''_{(2)}.
\end{align*}
By the induction hypothesis, we have
\begin{align*}
(\epl\ot\id)\col(F')=\beta_{l}(F'),&\quad
(\epl\ot\id)\col(F'')=\beta_{l}(F'').
\end{align*}
That is
\begin{align}
\sum_{(F')}\epl(F'_{(1)})\ot F'_{(2)}=1_{k}\ot F',&\quad\sum_{(F'')}\epl(F''_{(1)})\ot F''_{(2)}=1_{k}\ot F''. \mlabel{eq:eassee1}
\end{align}
Thus
\begin{align*}
(\epl\ot\id)\col(F)&=(\epl\ot\id)\col(F'F'')\\
&=(\epl\ot\id)(\col(F')\col(F''))\quad\text{(by Lemma ~\mref{lem:colh})}\\
&=(\epl\ot\id)\bigg(\sum_{(F')(F'')}F'_{(1)}F''_{(1)}\ot F'_{(2)}F''_{(2)}\bigg)\\
&=\sum_{(F')(F'')}\epl(F'_{(1)})\epl(F''_{(1)})\ot F'_{(2)}F''_{(2)}\quad\text{(by Lemma ~\mref{lem:eplh})}\\
&=\bigg(\sum_{(F')}\epl(F'_{(1)})\ot F'_{(2)}\bigg)\bigg(\sum_{(F'')}\epl(F''_{(1)})\ot F''_{(2)}\bigg)\\
&=(1_{k}\ot F')(1_{K}\ot F'')\quad\text{(by Eq.~\meqref{eq:eassee1})}\\
&=1_{k}\ot F'F''\\
&=\beta_{l}(F).
\end{align*}
This completes the inductions on the depth and breadth.
 \end{proof}

With all the preparation, we arrive at the first main result of this section. 

\begin{theorem}\mlabel{thm:bia}
The quintuple $(\hrts,\mul,1,\col,\epl)$ is a unitary bialgebra.
\begin{proof}
We have shown that the triple $(\hrts,\mul,1)$ is a unitary algebra. By Lemma ~\mref{lem:coass2} and Lemma ~\mref{lem:coal}, we know that the triple $(\hrts,\col,\epl)$ is a coalgebra. Then, applying Lemmas ~\mref{lem:colh} and ~\mref{lem:eplh}, we conclude that $(\hrts,\mul,1,
\col,\epl)$ is a unitary bialgebra.
\end{proof}
\end{theorem}

\subsection{A Hopf algebra structure on $(X, \Omega)$-decorated rooted trees}
\mlabel{ss:hopf}
We first recall some background on connected filtered bialgebras and Hopf algebras, following \mcite{GG}, from which we refer the reader for details and other references.

A bialgebra $(H,m,u,\Delta,\varepsilon)$ is called {\bf graded} if there are {\bfk}-submodules $H^{(n)}, n\geq0$, of $H$ such that
\begin{enumerate}
    \item $H=\bigoplus\limits^{\infty}_{n\geq0}H^{(n)},$
    \item $H^{(p)}H^{(q)}\subseteq H^{(p+q)},$
    \item $\Delta(H^{(n)})\subseteq\bigoplus\limits^{}_{p+q=n}H^{(p)}\otimes H^{(q)}, \quad  n, p, q\geq0.$
\end{enumerate}
Elements of $H^{(n)}$ are said to have degree $n$. The bialgebra $H$ is called {\bf connected graded} if in addition $H^{(0)}=\mathrm{im}~ u~(={\bfk})$ and   $\mathrm{ker}~ \varepsilon=\bigoplus_{n\geq 1}H^{(n)}$.
It is well known that a connected graded bialgebra is a Hopf algebra.

\nc{\hgrad}{H_{{\rm RT}}^{(n)}}
\nc{\hfilt}{H_{{\rm RT},(n)}}

The degree $|F|$ of $F\in \calf_{{\rm RT}}(X,\Omega)$ is the number of vertices of $F$. For $n\geq 0$, let $\calf^{(n)}$ denote the set of $F\in \calf_{{\rm RT}}(X,\Omega)$ with degree $n$. Let $\hrts^{(n)}:=\hgrad:=\bfk \calf^{(n)}$ to be the graded space. Since the grading defined above does not satisfy condition $(c)$ for connected graded bialgebras under $\col$, H is not a connected graded bialgebra. Instead, we consider the associated filtration 
\begin{equation}
\hfilt:=\bigoplus_{k\leq n}H_{{\rm RT}}^{(k)},
\mlabel{eq:cofil}
\end{equation} 
and prove that the bialgebra $\hrts$ is a Hopf algebra from the perspective of connected cofiltered bialgebras.

\begin{defn}\cite[Definition 2.2 and 2.8]{GG} \mlabel{def:fil}
\begin{enumerate}
\item  A coalgebra $(C, \Delta, \vep)$ is called {\bf coaugmented} if there is a linear map $u: \bfk \rightarrow C$, called the {\bf coaugmentation}, such that $\vep u=\id_\bfk$.
\item
A coaugmented coalgebra $(C,u,\Delta,\vep)$ is called {\bf cofiltered} if there is an exhaustive increasing filtration $\{C_{(n)}\}_{n\geq 0}$ of $C$ such that
\begin{equation}
\im\,u\subseteq C_{(n)}, \quad \Delta(C_{(n)})\subseteq\sum\limits^{}_{p+q=n}C_{(p)}\otimes C_{(q)}, \quad n\geq0,  p,q\geq0.
\mlabel{eq:cofill}
\end{equation}
Elements in $C_{(n)}\setminus C_{(n-1)}$ are said to have degree $n$. 
\item $C$ is called {\bf connected (filtered)} if in addition $C_{(0)}=\mathrm{im}\, u~(={\bfk})$.
\mlabel{def:b}
\end{enumerate}
\end{defn}

By the coaugmented condition, we have $C=\im\, u \oplus \ker \vep$. Then from $\im\,u \subseteq C_{(n)}$ and modularity, we have $C_{(n)}=\im\, u\oplus (C_{(n)} \cap \ker \vep)$, as stated in~\mcite{GG}.

\begin{lemma}\cite[Theorem 3.4]{GG}\mlabel{lem:filterhopf}
Let $H=(H,m,u,\Delta,\vep)$ be a bialgebra such that $(H, \Delta,\vep)$ is a connected cofiltered coaugmented coalgebra. Then $H$ is a Hopf algebra.
\end{lemma}
We now obtain the main result of this section.

\begin{theorem}\mlabel{thm:hrth}
 The quintuple $(\hrts,\mul,1,\col,\epl)$, equipped with the filtration given in Eq.~\meqref{eq:cofil}, is a connected cofiltered coaugmented coalgebra and consequently a Hopf algebra.
\end{theorem}
 
\begin{proof}
The definition of $\epl$ shows that $(\hrts,1,\col,\epl)$ is a coaugmented coalgebra. To prove that $(\hrts,1,\col,\epl)$ is cofiltered, take $F\in \calf^{(k)}, k\geq 0$. By the definition of $\col(F)$ from Eq.~\meqref{eq:ccd}: 
$$ \col(F)=\sum_{F'\in\calf_{F}}F'\ot\widetilde{F/F'},$$
the element $\widetilde{F/F'}$ is in $H_{{\rm RT},(k-|F'|)}$. Thus 
$\col(F)$ is in $\sum_{p+q=k}H_{{\rm RT},(p)}\ot H_{{\rm RT},(q)}.$
Therefore, 
$$ \col(H_{{\rm RT},(n)}) \subseteq \sum_{k\leq n}  \bfk \col(\calf^{(k)}) \subseteq \sum_{k\leq n} \sum_{p+q=k} H_{{\rm RT},(p)}\ot H_{{\rm RT},(q)} = \sum_{p+q=n} H_{{\rm RT},(p)}\ot H_{{\rm RT},(q)}.$$

We also observe that $H_{{\rm RT},(0)}=\bfk$. Therefore, $\hrts$ is a connected cofiltered coaugmented coalgebra. Thus $\hrts$ is a Hopf algebra by Lemma~\mref{lem:filterhopf}.
\end{proof}

\subsection{Free \texorpdfstring{$\Omega$}{Omega}-cocycle Hopf algebras of $(X, \Omega)$-decorated rooted forests}
\mlabel{ss:free}
Thanks to the degree and grafting operators defined above, we can endow $\hrts$ with additional algebraic structures through the notion of operated algebras.
\begin{defn}
\mcite{Guo09} \mlabel{de:ofil}
    Let $\Omega$ be a nonempty set.
	\begin{enumerate}
		\item
		An {\bf $\Omega$-operated algebra} is an algebra $A$ together with a family of linear operators $P_{\omega}: A\to A$, $\omega\in \Omega$.
		\item
		Let $(A,\, \Po)$ and $(A',\,\Pop)$ be  $\Omega$-operated algebras.
		A linear map $f : A\rightarrow A'$ is called an {\bf $\Omega$-operated algebra homomorphism} if $f$ is an algebra homomorphism such that $f P_\omega = P'_\omega f$ for $\omega\in \Omega$.
		\item
		A {\bf free $\Omega$-operated algebra on a set $X$} is an $\Omega$-operated algebra $(A, \Po)$ together with a set map $j_{X}: X\rightarrow A$ with the property that, for any $\Omega$-operated algebra $(A', \Pop)$ and any set map $f: X\rightarrow A'$, there is a unique homomorphism $\bar f:A\to A'$ of $\Omega$-operated algebras such that $\bar{f} j_X=f$.
\item An $\Omega$-operated algebra $(R,P_\Omega)$ with a grading $R=\oplus_{n\geq 0} R^n$ (resp. an increasing filtration $\{R_n\}_{n\geq 0}$) is called an {\bf $\Omega$-operated graded algebra} (resp. {\bf $\Omega$-operated filtered algebra}) if $(R,\oplus_{n\geq 0}R^n)$ is a graded algebra (resp. $(R,\{R_n\}_{n\geq 0})$ is a filtered algebra) and
\begin{equation*}
P_\omega (R^n)\subseteq R^{n+1} \quad
\text{(resp. } P_\omega (R_n)\subseteq R_{n+1}).
\mlabel{eq:ofil}
\end{equation*}
	\end{enumerate}
\end{defn}
Note that the grafting operator $\pl$ increases the degree of a rooted forest by one. 

\begin{lemma}\cite[Lemma 2.5]{ZGG},\mcite{ZGG16}
\mlabel{lem:ofil}
Let $X$ be a set and $\Omega$ be a nonempty set. 
\begin{enumerate}
\item The algebra $\hrts$ with the family of grafting operator $B_{\omega}^+$,\, $\omega\in\Omega$, is an $\Omega$-operated algrbra. 
\item The $\Omega$-operated algebra $\hrts$ with its grading $\hrts=\oplus_{n\geq 0}H_{{\rm RT}}^{(n)}$ and the associated filtration $\{H_{{\rm RT},(n)}:=\oplus_{k\leq n}H_{{\rm RT}}^{(k)}\}_{n\geq 0}$ is an $\Omega$-operated graded algebra and an $\Omega$-operated filtered algebra.
\item 
Let $j_{X}: X\hookrightarrow \hrts$, $x \mapsto \sbullet_{x}$ be the natural embedding and $m_{{\rm RT}}$ be the concatenation product.
The quadruple $(\hrts, \,\mul,\,1, \, \Bo)$ together with $j_X$ is the free $\Omega$-operated algebra on $X$.
\mlabel{lem:propm}
\end{enumerate}
\end{lemma}

We introduce the following generalization of various cocycle bialgebras.

\begin{defn}
\begin{enumerate}
\item An {\bf $\Omega$-operated bialgebra} is a bialgebra $(A,m,1_A,\Delta, \varepsilon)$ which is also an $\Omega$-operated algebra $(A,\,\Po)$.
\item Let $(A,\,\Po)$ and $(A',\,\Pop)$ be $\Omega$-operated bialgebras.
A linear map $f : H\rightarrow H'$ is called an {\bf $\Omega$-operated bialgebra homomorphism} if $f$ is a bialgebra homomorphism such that $f  P_\omega = P'_\omega f$ for $\omega\in \Omega$.
\item
An {\bf$\Omega$-cocycle bialgebra} is an $\Omega$-operated bialgebra $(A,m,1_A,\Delta, \varepsilon,\,\Po)$ which satisfies the Hochschild 1-cocycle condition of weight $\lambda_\Omega$:
\begin{equation}
\Delta P_\omega=P_\omega\ot 1_H +\lambda_\omega\id\ot1_H + (\id\ot P_\omega)\Delta \, \text{ for }\omega \in \Omega,
\mlabel{eq:cocycle2}
\end{equation}
where $\lambda_\Omega:=(\lambda_\omega)_{\omega\in \Omega}\subseteq \bfk$
be a set of scalars parameterized by $\Omega$.
If the bialgebra in an $\Omega$-cocycle bialgebra is a Hopf algebra, then it is called an {\bf$\Omega$-cocycle Hopf algebra}.
\item
The {\bf free $\Omega$-cocycle bialgebra on a set $X$} is an $\Omega$-cocycle bialgebra $(A_X,\, m_X,\, 1_X,\, \Delta_X,\,  \varepsilon_X,\,\\ \Po)$ together with a set map $j_X:X\to A_X$ with the property that, for any $\Omega$-cocycle bialgebra $(A,\, m,\, 1_A,\,  \Delta, \, \varepsilon,\,\Pop)$ and set map $f:X\to A$ (satisfies $\Delta(f(x))=f(x)\ot 1_A+1_A\ot f(x)+\mu_x 1_A\ot 1_A$), there is a unique homomorphism $\free{f}:A_X\to A$ of $\Omega$-operated bialgebras such that $\free{f} j_X=f$. The concept of a {\bf free $\Omega$-cocycle Hopf algebra} is defined in the same way.
\mlabel{it:def4}
\end{enumerate}
\mlabel{de:decHopf}
\end{defn}

Lemma~\mref{lem:ofil}.\meqref{lem:propm} suggests that the structure of the $\Omega$-operator algebra $\hrts$ can be extended to a bialgebra. For this purpose, we first prove a lemma.

\begin{lemma} \mlabel{lem:coidealv}
    Let $(H,\, m,\,1_H,\, \Delta,\,\varepsilon,\, \Po)$
be an $\Omega$-cocycle bialgebra. Then
$$\varepsilon P_{\omega}(h)=-\lambda_\omega\varepsilon(h),$$
for all $h\in H$ and $\omega\in \Omega$.
\end{lemma}
\begin{proof}
Let $\omega\in\Omega$ and $h\in H$. Applying $\varepsilon=(\varepsilon\ot \varepsilon)\Delta$, we have 
\begin{align*}
\varepsilon P_{\omega}(h)&
=(\varepsilon\ot \varepsilon) \Delta(P_{\omega}(h))\\
&=(\varepsilon\ot \varepsilon)\Big(P_{\omega}(h)\ot 1_H+\lambda_\omega h\ot1_H+(\id \ot P_{\omega})\Delta(h)\Big)\quad(\text{by Eq.~\meqref{eq:cocycle2}})\\
&=(\varepsilon\ot \varepsilon)(P_{\omega}(h)\ot 1_H)+(\varepsilon\ot \varepsilon)(\lambda_\omega h \ot1_H)+(\varepsilon\ot \varepsilon P_{\omega})\Delta(h)\\
&=\varepsilon P_{\omega}(h)+\lambda_\omega\varepsilon(h)+\sum_{(h)}\varepsilon(h_{(1)})\varepsilon P_{\omega}(h_{(2)})\\
&=\varepsilon P_{\omega}(h)+\lambda_\omega\varepsilon(h)+\varepsilon P_{\omega}\bigg(\sum_{(h)}\varepsilon(h_{(1)})(h_{(2)})\bigg)\\
&=\varepsilon P_{\omega}(h)+\lambda_\omega\varepsilon(h)+\varepsilon P_{\omega}(h).
\end{align*}
Then we have $\varepsilon P_{\omega}(h)+\lambda_\omega\varepsilon(h)=0$ and $\varepsilon P_{\omega}(h)=-\lambda_\omega\varepsilon(h)$.
\end{proof}

With the above lemma, we can now prove the following theorem, which generalizes the universal properties of various cocycle Hopf algebras~\mcite{Mo,ZGG}.
\begin{theorem}
Let $j_{X}: X\hookrightarrow \hrts$, $x \mapsto \sbullet_{x}$ be the nature embedding.
\begin{enumerate}
\item
 The sextuple $(\hrts, \,\mul,\,1, \, \Delta_{{\rm RT}},\,\epl, \,\Bo)$ together with $j_X$ is the free $\Omega$-cocycle  bialgebra on $X$.  \mlabel{it:fubialg}

\item The Hopf algebra given by the connected cofiltered coaugmented bialgebra $$(\hrts,\,\mul,\,1, \, \Delta_{{\rm RT}},\,\epl, \, \Bo)$$ 
together with $j_X$ is the free $\Omega$-cocycle Hopf algebra on $X$.
 \mlabel{it:fuhopf}
\end{enumerate}
\mlabel{thm:propm}
\end{theorem}

\begin{proof}
(\mref{it:fubialg})
By Theorem~\mref{thm:bia} and Lemma~\mref{lem:ofil}, the quintuple $(\hrts, \mul,\,1,\, \Delta_{{\rm RT}}, \epl)$ is a $\Omega$-operated bialgebra.
Additionally, by Eq.~\meqref{eq:cdbp}, the sextuple $(\hrts,
\mul,\,1,\, \Delta_{{\rm RT}},\epl, \,\Bo)$ is an $\Omega$-cocycle  bialgebra.

Fix an $\Omega$-cocycle bialgebra $(H,\, m,\,1_H,\, \Delta,\,\varepsilon,\, \Po)$ and a set map $f: X\rightarrow H$ satisfying
\begin{align}
\Delta(f(x))=f(x)\ot 1_H+\mu_x 1_H\ot1_H+1_H\ot f(x) \quad \text{for all}\ x\in X.
\mlabel{eq:prim}
\end{align}
Obviously, $(H,\, m,\, 1_H,\, \Po)$ is an $\Omega$-operated algebra.
By Lemma~\mref{lem:ofil}.\meqref{lem:propm}, there exists a unique $\Omega$-operated algebra homomorphism $\free{f}:\hrts \to H$ such that $\free{f} j_X={f}$. We only need to show that $\free{f}$ is a coalgebra homomorphism, that is for all $F\in \calf_{{\rm RT}}(X,\Omega)$, we have
\begin{align}
\Delta \free{f} (F)&=(\free{f}\ot \free{f}) \col (F), \mlabel{eq:copcomp}\\
\varepsilon \free{f} (F)&=\epl (F).
\mlabel{eq:counit}
\end{align}

To prove Eq.~\meqref{eq:copcomp}, we define the set
\begin{align*}
\mathscr{A} := \{ F\in \hrts \, | \, \Delta(\free{f}(F))=(\free{f} \ot \free{f}) \col(F) \}.
\end{align*}
By Lemma~\mref{lem:ofil}.\meqref{lem:propm}, $\hrts$ is an $\Omega$-operated algebra generated by $X$. Thus Eq.~\meqref{eq:copcomp} holds if $\mathscr{A}$ is proved to be an $\Omega$-operated subalgebra of $\hrts$ containing $X$.

Since $\free{f}$ is an $\Omega$-operated algebra homomorphism, and $\col$ and $\Delta$ are algebra homomorphisms on $\hrts$ and $H$ respectively, it follows that $1 \in \mathscr{A}$ and $\mathscr{A}$ is a subalgebra of $\hrts$. Moreover, For any generator $x \in X$, we have
\begin{align*}
\Delta(\free{f}(\sbullet_{x}))=&\ \Delta(f(x))\\
=&\ f(x) \ot 1_H +\mu_x 1_H\ot1_H+1_H \ot f(x)  \quad \text{(by Eq.~\meqref{eq:prim})}\\
=&\ \free{f} (\sbullet_{x})\ot \free{f}(1)+\mu_x\free{f}(1)\ot\free{f}(1)+\free{f}(1)\ot \free{f}(\sbullet_{x}) \\
=&\ (\free{f} \ot \free{f})(\sbullet_{x} \ot 1 +\mu_x 1\ot1+1 \ot \sbullet_{x} )\\
=&\ (\free{f} \ot \free{f})\col(\sbullet_{x}).
\end{align*}
Thus $\sbullet_{x} \in \mathscr{A}$.

Further for any $F \in \mathscr{A}$ and $\omega \in \Omega$, we have
\begin{align*}
&\ \Delta \free{f} (\pl({F}))\\
=&\ \Delta P_\omega(\free{f} ({F}))
\quad(\text{by $\bar{f}$ being an $\Omega$-operated algebra homomorphism}) \\
=&\ P_\omega(\free{f}({F}))\ot 1_{H}+\lambda_\omega\free{f}(F)\ot1_H+ (\id\ot P_\omega)\Delta(\free{f} ({F}))\quad(\text{by Eq.~\meqref{eq:cocycle2}})\\
=&\ P_\omega(\free{f}({F}))\ot 1_{H}+\lambda_\omega\free{f}(F)\ot1_H+ (\id\ot P_\omega)(\free{f}\ot \free{f}) \col ({F}) \quad \text{(by $F \in \mathscr{A}$)}\\
=&\ P_\omega(\free{f}({F}))\ot 1_{H}+\lambda_\omega\free{f}(F)\ot1_H+ (\free{f}\ot P_\omega\free{f}) \col ({F})\\
=&\ \free{f}(\pl({F}))\ot 1_{H}+\lambda_\omega\free{f}(F)\ot1_H+ (\free{f}\ot \free{f}B_{\omega}^+) \col ({F}) \\ &(\text{by $\bar{f}$ being an $\Omega$-operated algebra homomorphism}) \\
=&\ (\free{f}\ot \free{f})\Big(\pl({F})\ot 1+\lambda_\omega F\ot1+(\id\ot B_{\omega}^+)\col ({F})\Big) \\
=&\ (\free{f}\ot \free{f}) \col (B_{\omega}^+({F})).
\end{align*}
Thus $\mathscr{A}$ is closed under  $\pl$ for $\omega \in \Omega$ and so $\mathscr{A}=\hrts$.

Similarly, to verify Eq.~\meqref{eq:counit}, we define the subset
\begin{align*}
\mathscr{B}:=\{ F \in \hrts \mid \, \varepsilon(\free{f}(F))= \varepsilon_{\mathrm{{\rm RT}}}(F)\} \subseteq \hrts,
\end{align*}
and verify that it is an $\Omega$-operated subalgebra of $\hrts$ containing $X$.

Since $\free{f}$ is an $\Omega$-operated algebra homomorphism, $\varepsilon_{\mathrm{RF}}$ and $\varepsilon$ are algebra homomorphisms on $\hrts$ and $H$ respectively. So we get $1 \in \mathscr{B}$ and $\mathscr{B}$ is a subalgebra of $\hrts$. For any generator $x \in X$, $f(x)$ satisfies Eq.~\meqref{eq:prim}.
Hence $\varepsilon(f(x))=-\mu_x 1_k$. Then by Eq.~\meqref{eq:dele2}, we have
\begin{align*}
\varepsilon(\free{f}(\sbullet_{x}))= \varepsilon(f(x))=-\mu_x 1_k=\varepsilon_{\mathrm{{\rm RT}}}(\sbullet_{x}).
\end{align*}
Thus $\sbullet_{x} \in \mathscr{B}$. For any $F \in \mathscr{B}$ and $\omega \in \Omega$, we have
\begin{align*}
\varepsilon(\free{f}(\pl(F)))&\ = \varepsilon(P_{\omega}(\free{f}(F))) \quad \text{(by $\free{f}$ being an $\Omega$-operated algebra homomorphism)}\\
&\ =-\lambda_\omega\varepsilon(\free{f}(F))\quad\text{(by Lemma ~\mref{lem:coidealv})}\\
&\ =-\lambda_\omega\epl(F)\quad\text{(by $F\in\mathscr{B}$)}\\
&\ =\varepsilon_{\mathrm{{\rm RT}}}(\pl(F)).
\end{align*}
Hence $\mathscr{B}$ is closed under $\pl$ for $\omega \in \Omega$ and so $\mathscr{B}= \hrts$. This completes the proof.

(\mref{it:fuhopf})
The proof follows from Item~(\mref{it:fubialg}) and the well-known fact that
any bialgebra homomorphism between two Hopf algebras is compatible with the antipodes~\cite[Lemma~4.04]{Sw}.
\end{proof}

\noindent
{\bf Acknowledgments.} This research is supported by the National Natural Science Foundation of China (Grant No.12101316 and 12301025).

\noindent
{\bf Declaration of interests.} The authors have no conflicts of interest to disclose.

\noindent
{\bf Data availability.} Data sharing is not applicable as no new data were created or analyzed.

\end{document}